\documentclass[smallcondensed, numbook, envcountsame]{svjour3}

\usepackage[a4paper, left=4cm, right=4cm, top=3cm, bottom=3cm]{geometry}
\usepackage{tkz-graph}
\usepackage{tkz-berge}
\usepackage{circuitikz}
\usetikzlibrary{arrows}
\usepackage{nicefrac}

\usepackage[latin1]{inputenc}
\usepackage{enumitem}
\usepackage{graphicx}

\usepackage{dsfont}
\usepackage{amssymb}
\usepackage{amsmath}

\usepackage[pdfborder={0 0 0}]{hyperref}

\allowdisplaybreaks[1] 

\renewcommand{\hat}{\widehat}
\renewcommand{\tilde}{\widetilde}

\DeclareMathOperator*{\dom}{dom}
\newcommand{\menge}[1]{\left\lbrace #1 \right\rbrace}
\newcommand{\mengedef}[2]{\left\lbrace #1  ~ | ~  #2 \right\rbrace}
\newcommand{\all}{\ \forall\ }
\newcommand{\ex}{\ \exists\ }
\newcommand{\restr}{\mathord\restriction}
\newcommand{\Fin}{\mathcal{F}in}
\newcommand{\domE}{\dom \mathcal{E}}

\def\N{\mathbb N}
\def\Z{\mathbb Z}
\def\R{\mathbb R}
\def\E{\mathbb{E}}
\def\P{\mathbb{P}}
\def\F{\mathcal{F}}
\def\lb{\left(}
\def\rb{\right)}
\def\ls{\left[}
\def\rs{\right]}

\begin{document}

\spdefaulttheorem{dfn}{Definition}{\bf}{\rm}
\spdefaulttheorem{rem}{Remark}{\bf}{\rm}
\spdefaulttheorem{lem}{Lemma}{\bf}{\rm}
\spdefaulttheorem{prop}{Proposition}{\bf}{\rm}
\spdefaulttheorem{cor}{Corollary}{\bf}{\rm}
\spdefaulttheorem{thm}{Theorem}{\bf}{\rm}
\spdefaulttheorem{exa}{Example}{\bf}{\rm}

\spnewtheorem*{xproof}{}{\itshape}{\rmfamily}
\newcommand\xprooftitle{}
\renewenvironment{proof}[1][\proofname]
{\renewcommand\xproofname{#1.}\xproof}
{\endxproof}

\smartqed

\title{A Characterization of Effective Resistance Metrics}
\author{Tobias Weihrauch}
\date{\today}

\institute{T. Weihrauch
	\at Universit\"{a}t Leipzig, Fakult\"{a}t f\"{u}r Mathematik und Informatik, Augustusplatz 10, 04109 Leipzig, Germany, \email{weihrauch@math.uni-leipzig.de}}

\maketitle

\begin{abstract}
	We produce a characterization of finite metric spaces which are given by the effective resistance of a graph. This characterization is applied to the more general context of resistance metrics defined by Kigami. A countably infinite resistance metric gives rise to a sequence of finite, increasing graphs with invariant effective resistance. We show that these graphs have a unique limit graph in terms of the convergence of edge weights and that their associated random walks converge weakly to the random walk on the limit graph. If the limit graph is recurrent, its effective resistance is identified as the initial resistance metric.
\end{abstract}
\subclass{Primary 05C12 $\cdot$ 05C81 $\cdot$ 31E05; Secondary 05C50 $\cdot$ 54E70 $\cdot$ 60J45}
\keywords{Weighted graph, Effective resistance, Graph Laplacian, Resistance metrics, Random walk, Weak convergence}

\section{Introduction}
\label{section:Introduction}

Consider an undirected, connected graph on a finite vertex set $V$ with no self-loops and no multiple edges. Introducing positive weights $c(x,y)$ on the edges and interpreting them as conductances, this graph becomes an electrical network which induces an effective resistance $R(x,y)$ between every two nodes $x,y \in V$. A well-established but somewhat surprising result is that the effective resistance is in fact a metric on $V$, see \cite{Tetali1991,Klein1993}. A very nice proof of this statement utilizing the connection between electrical currents and random walks on graphs (see \cite{DoyleSnell1984,LyonsPeres2015}) is given in \cite{Tetali1991} by Tetali. An integral part of this proof is to represent the effective resistance $R(x,y)$ using the expected number of times a random walk starting in $x$ visits $x$ before reaching $y$, more precisely
\begin{equation}
\label{eq:IntroductionEffRes}
R(x,y) = \frac{1}{c_x}\E_x\left[\sum_{k=0}^{\tau_y-1} \mathds{1}_x(\omega_k)\right]
\end{equation}
where $c_x$ is the sum of all edge weights attached to $x$. 

Since every reversible, irreducible Markov chain on a finite state space $V$ which leaves its current state with probability 1 is a random walk on some weighted graph $G = (V,c)$, see \cite{AldousFill2002}, this also produces a natural way to obtain a metric from such a stochastic process. Furthermore, a metric $d$ on $V$ admits a representation as in (\ref{eq:IntroductionEffRes}) if and only if it is the effective resistance of a weighted graph on $V$. This leads to the following questions which we seek to answer in this paper. Which finite metric spaces $(V,d)$ are given by effective resistances of graphs? Does there exist a concise condition in terms of the metric $d$? Can we extend such a condition to countably infinite metric spaces? Finally, when can the representation (\ref{eq:IntroductionEffRes}) be extended to countably infinite metric spaces?

In Section \ref{section:FiniteERS} we produce the following necessary and sufficient condition for a finite metric space $(V,d)$ to be the effective resistance of a graph (Theorem \ref{theorem:CharacterizationFiniteERS}). For $x,y,z \in V$, let
$$A_y(x,z) := \begin{cases}
	1 & , ~ x = y = z\\
	\frac{1}{2}(d(x,y) + d(y,z) - d(x,z)) & , \text{ otherwise}
\end{cases}$$
and $b_y(x) = 1 - \delta_y(x)$. Then, there exists a graph $G$ with effective resistance $d$ if and only if $\det A_y > 0$ for some $y \in V$ and the unique solution matrix $c \in \R^{V \times V}$ of the family of linear equation systems
$$A_y \cdot c(\cdot , y) = b_y ~ , ~ y \in V$$
has only non-negative entries. In this case, $G = (V,c)$.

Section \ref{section:ResistanceMetrics} is devoted to recalling the theory of resistance metrics introduced by Kigami in \cite{Kigami2001} and applying Theorem \ref{theorem:CharacterizationFiniteERS} to it in order to obtain a characterization of countably infinite resistance metrics (see Theorem \ref{theorem:CharResistanceMetrics}). 

In Section \ref{section:LimitingProperties} we are given a countably infinite resistance metric $(V,R)$. This induces a sequence of growing finite graphs $(V_n,c_n)$ with invariant effective resistance $R\restr_{V_n}$. We investigate the limiting behavior of the edge weights $c_n$. The main result of this section (Theorem \ref{theorem:KantenLimesEx}) states that for fixed vertices $x,y \in V$, the sequence $c_n(x,y)$ is monotonically decreasing and thus has a limit which leads to the definition of a limit graph $G_R$ for $R$. While $G_R$ may not always admit a well-defined Laplacian, it is shown in Proposition \ref{prop:LimitGraphConsistency} that if $R$ is an effective resistance of some countably infinite graph $G$, one has $G_R = G$.

In Section \ref{section:WeakLimitRandomWalk} we utilize Prohorov's theorem to show that the random walks on these graphs have a weakly convergent subsequence if and only if the sequence of edge weights converges in a well-behaved manner (Theorem \ref{theorem:WeakConvergence}). This section's main result (Theorem \ref{theorem:WeakLimitIsRandomWalkOnLimitGraph}) states that in this case, the whole sequence is weakly convergent and the weak limit is identified as the random walk of $G_R$. Furthermore, we show that if $G_R$ is recurrent, $R$ admits a probabilistic representation as in (\ref{eq:IntroductionEffRes}) using the limit random walk (Theorem \ref{theorem:ProbabilisticRepresentationMetricInfinite}). Since such a representation is known to exist for the effective resistance of a recurrent graph, it follows that the effective resistance of $G_R$ is exactly $R$ (Corollary \ref{cor:EffResOfRecurrentLimitGraph}).

The notion of effective resistance has been extensively studied before by utilizing various methods. Besides potential theory \cite{Soardi1994}, the notion is also tightly connected to other fields of mathematical study, including random walks and reversible Markov chains \cite{DoyleSnell1984,Chandra1997}, the theory of reproducing kernels \cite{JorgensenPearse2009,Kasue2010} and algebraic graph theory \cite{DorflerBullo2013}. A comprehensive study of the probabilistic approach to electrical currents on graphs is given in the book \cite{LyonsPeres2015}.

Classically, effective resistance is defined using solutions of Kirchoff's laws which correspond to harmonic functions \cite{DoyleSnell1984}. When trying to extend the definition to infinite graphs or sets, problems arise because of the non-uniqueness of such solutions. This leads to several different notions like the \emph{free} and \emph{wired} effective resistance \cite{LyonsPeres2015}. The most general theory which allows for \emph{resistance metrics} on arbitrary sets is introduced by Kigami in \cite{Kigami2001} using \emph{resistance forms}.

Since there are so many different approaches to effective resistances and notational conventions vary a lot, we use the remainder of this section to introduce some basic definitions and notation used in this work.

\paragraph{Graphs}
A \emph{weighted graph} $G$ is a pair $(V,c)$ consisting of a finite or countably infinite set of \emph{vertices} $V \neq \emptyset$ and a \emph{weight function} $c:V \times V \to \R_{\geq 0}$ such that $c(x,x) = 0$ and $c(x,y) = c(y,x)$ for all $x,y \in V$. Furthermore, for $x \in V$, let $c_x := \sum_{y \in V} c(x,y)$. We consider two vertices $x,y$ to be adjacent if $c(x,y) > 0$. If not explicitly stated otherwise, we assume every occurring graph to be weighted, connected and locally finite in the sense that $c_x < \infty$ for all $x \in V$.

\paragraph{Laplacian}
For a locally finite graph $G = (V,c)$, a function $f: V \to \R$ and $x \in V$, let
\begin{equation}
\label{eq:GraphLaplacian}
(\Delta f)(x) = f(x) - \sum_{y \in V} \frac{c(x,y)}{c_x} f(y).
\end{equation}
We say a function $f: V \to \R$ is \emph{harmonic} if $\Delta f \equiv 0$.

We will often switch seamlessly between the operator $\Delta$ and its associated matrix which is defined by
$$\Delta(x,y) = \delta_x(y) - \frac{c(x,y)}{c_x}.$$
The operator $\Delta$ is a normalized, non-negative version of what is called the \emph{Laplacian} of $G$ (cf. \cite{JorgensenPearse2009,Kumagai2014,LyonsPeres2015}). It should not be confused with another normalized version which is used in spectral graph theory, see, e.g., \cite{Chung1997,Spielman2007}, although the spectra of both matrices are equal.

Note that $G$ does not need to be connected for $\Delta$ to be well-defined as long as $0 < c_x < \infty$ for all $x \in V$.

\paragraph{Effective resistance}
Let $G = (V,c)$ be a finite graph and $x,y \in V$. Interpreting $c(x,y)$ as the pair-wise conductance between two nodes $x$ and $y$, $G$ represents an electrical network. The \emph{effective resistance} $R(x,y)$ of $G$ between the nodes $x$ and $y$ can now be defined as the voltage drop between $x$ and $y$ when a unit current flows from $x$ to $y$ through $G$. More precisely, we have
\begin{equation}
\label{eq:DefinitionEffectiveResistance}
R(x,y) := \phi^{xy}(x)
\end{equation}
where $\phi^{xy}$ is the solution of the discrete Dirichlet problem
\begin{align*}
\label{eq:DirichletProblem}
\Delta \phi^{xy} & = \frac{1}{c_x} \mathds{1}_x - \frac{1}{c_y} \mathds{1}_y\tag{D}\\
\phi^{xy}(y) & = 0 ~.
\end{align*}
Note that there are many different ways to define effective resistance. An effort to show the equivalence of some of these different definitions can be found in \cite[Theorem 2.3]{JorgensenPearse2009}.

\paragraph{Random walk}
Let $\N_0 = \N \cup \menge{0}$. For a vertex $x \in V$, consider the random walk starting in $x$ and jumping from vertex $y$ to $z$ with probability $c(y,z)/c_y$. We will denote its distribution on the space of trajectories $\Omega = V^{\N_0}$ by $\P_x$. Hence, $\P_x$ is an irreducible, reversible Markov chain such that for $y,z \in V$, we have
\begin{equation}
\label{eq:DefinitionRandomWalk}
\P_x[\omega_0 = x] = 1 ~ , ~ \P_x[\omega_{k+1} = z ~ | ~ \omega_k = y] = \frac{c(y,z)}{c_y} ~.
\end{equation}
We denote by $\E_x$ the expectation of $\P_x$, i.e. $\E_x[f] = \int_{\Omega} f ~ \text{d}\P_x$.

For $x \in V$ and $\omega \in \Omega$, let
\begin{align}
\label{eq:DefinitionHittingTime}
\tau_x(\omega) & = \inf\mengedef{k \geq 0}{\omega_k = x} \text{ and}\\
\tau^+_x(\omega) & = \inf\mengedef{k \geq 1}{\omega_k = x}.\notag
\end{align}
be the \emph{hitting time} of $x$.

\paragraph{Metric space}
For any set $X$ and a mapping $d: X \times X \to \R_{\geq 0}$, the pair $(X,d)$ is a metric space if $d$ satisfies the following conditions for all $x,y,z \in X$.
\begin{align*}
d(x,y) = 0 & \Leftrightarrow x = y \label{eq:M1}\tag{M1}\\
d(x,y) & = d(y,x) \label{eq:M2}\tag{M2}\\
d(x,y) + d(y,z) & \leq d(x,z) \label{eq:M3}\tag{M3}
\end{align*}
(\ref{eq:M3}) is called \emph{triangle inequality}.

\paragraph{Finite exhaustion}
For a countably infinite set $V$, a \emph{finite exhaustion} is a sequence $(V_n)_{n \in \N}$ of subsets of $V$ such that $V_n \subseteq V_{n+1}$, $|V_n| < \infty$ for all $n\in \N$ and $\bigcup_{n \in \N} V_n = V$. 

If we consider (countably) infinite graphs $(V,c)$, we will encounter sums of the form $\sum_{v \in V} f(v)$. Whenever that happens, we will implicitly assume that we have some given exhaustion $(V_n)_{n \in \N}$ and define
$$\sum_{v \in V} f(v) := \lim_{n \to \infty} \sum_{v \in V_n} f(v)$$
if the right hand-side exists. Note that if $f(v) \geq 0$ for all $v \in V$, then the right hand-side is independent of the choice of $(V_n)_{n \in\N}$ since it either converges absolutely or is infinite.

\paragraph{Matrix restriction}
For a matrix $A  \in \R^{V \times W}$ and $V' \subseteq V, W' \subseteq W$, let $A\restr_{V' \times W'}$ be the matrix which results from $A$ by restricting the rows of $A$ to indices in $V'$ and the columns to indices in $W'$. If $V' = W'$, we will write $A \restr_{V'}$.

Furthermore, let $A^t$ denote the transpose of $A$.

\section{Finite effective resistance spaces}
\label{section:FiniteERS}
Throughout this section let $(V,d)$ be a finite metric space. We say $(V,d)$ is an \emph{effective resistance space (ERS)} if there exists a graph $G = (V,c)$ with effective resistance $d$.

Our goal is to find a necessary and sufficient condition for when $(V,d)$ is an ERS. Therefore, we first establish some properties of such spaces. Building on \cite{Tetali1991}, we show that the underlying graph $G = (V,c)$ of an ERS can be reconstructed from $d$ by means of a family of linear equation systems
$$A_y \cdot c(\cdot, y) = b_y ~, y \in V$$
which use triangle inequality defects of $d$ as coefficients (Proposition \ref{prop:CIsSolutionOfLES}) and that $\det A_y > 0$ always holds (Proposition \ref{prop:MatrixOfTriangleDefectsHasPositiveDeterminant}). Furthermore, by definition of an ERS, $c(x,y)$ has to be non-negative for all $x,y \in V$.

Conversely, if one assumes that $\det A_y > 0$ for an arbitrary metric space $(V,d)$, solving the corresponding linear equation systems yields a possible candidate for $c$. It turns out that assuming non-negativity of all $c(x,y)$ is then sufficient for $(V,c)$ to define a graph (Proposition \ref{prop:Symmetry}) which has effective resistance $d$ (Proposition \ref{prop:CAndMyAreCompatible}).

The statements of Theorem 1 and Corollary 3 in \cite{Tetali1991} can be generalized to fit our context of weighted graphs and are merged in the following proposition.
\begin{prop}[Tetali]
	Let $G = (V,c)$ be a finite graph with effective resistance $R$ and $x,y,z \in V$ such that $x \neq y$. Furthermore, let $\phi^{xy}$ be the solution of the Dirichlet problem (\ref{eq:DirichletProblem}). Then, the following holds.
	\begin{equation}
	\label{eq:PotentialTetali}
	\phi^{xy}(z) = \frac{1}{2}(R(x,y) + R(y,z) - R(x,z)) = \frac{1}{c_z} \E_x\left[\sum_{k=0}^{\tau_y-1} \mathds{1}_z(\omega_k)\right].
	\end{equation}
\end{prop}
In particular, (\ref{eq:PotentialTetali}) implies that
\begin{equation}
\label{eq:EffectiveResistanceProabilistic}
R(x,y) = \phi^{xy}(x) = \frac{1}{c_x}\E_x\left[\sum_{k=0}^{\tau_y-1} \mathds{1}_x(\omega_k)\right]
\end{equation}
which gives us a probabilistic representation of effective resistances. Furthermore, we can use (\ref{eq:PotentialTetali}) to prove that the effective resistance of a possibly unconnected graph exists if and only if this graph is connected.
\begin{lem}
	\label{lemma:PotentialIffGraphConnected}
	Let $G = (V,c)$ be a finite and possibly unconnected graph. Then, the Dirichlet problem ($\ref{eq:DirichletProblem}$) on $G$ has a solution for all $x,y \in V$ if and only if $G$ is connected. In this case, the solution is unique.
\end{lem}
\begin{proof}
	Suppose that $G$ is connected, let $x,y \in V$, $x \neq y$ and define
	$$\phi(z) := \frac{1}{c_z} \E_x\left[ \sum_{k=0}^{\tau_y - 1} \mathds{1}_z(\omega_k) \right].$$
	We leave the verification that this function indeed solves (\ref{eq:DirichletProblem}) as an exercise. For two solutions $\phi$ and $\phi'$ of (\ref{eq:DirichletProblem}), $h := \phi - \phi'$ satisfies $\Delta h \equiv 0$ and $h(y) = 0$. It is a standard result that harmonic functions on finite, connected graphs are unique (cf. \cite{LyonsPeres2015}).
	
	Now suppose that $G$ is not connected and $X \subseteq V$ is a connected component of $G$. For all $x \in X$, we then have $c_x = \sum_{y \in X} c(x,y)$ since $c(x,y) > 0$ implies $y \in X$. Hence, for any function $\phi: V \to \R$, we compute
	\begin{align*}
	\sum_{x \in X} (\Delta \phi)(x) \cdot c_x & = \sum_{x \in X} \left(c_x \phi(x) - \sum_{y \in V} \lb c(x,y) \phi(y) \rb \right) \\
	& = \lb \sum_{x \in X} c_x \phi(x)\rb - \sum_{x \in X} \sum_{y \in X}\lb c(x,y) \phi(y)\rb\\
	& = \lb\sum_{x \in X} c_x\phi(x)\rb - \sum_{y \in X} \left( \phi(y) \sum_{x \in X} c(x,y)\right)\\
	& = \lb \sum_{x \in X} c_x \phi(x)\rb - \sum_{y \in X} c_y \phi(y) = 0.
	\end{align*}
	Now let $x \in X$ and $y \in V \setminus X$. If $\phi: V \to \R$ satisfies $(\Delta \phi)(x) = \nicefrac{1}{c_x}$, then the above computation shows that there exists some $x' \in X$ such that $(\Delta \phi)(x') < 0$ and thus $\phi$ cannot solve $\Delta \phi = \nicefrac{1}{c_x}\mathds{1}_x - \nicefrac{1}{c_y} \mathds{1}_y$. Hence, there does not exist a solution for (\ref{eq:DirichletProblem}).\qed
\end{proof}
Note that if $(V,d)$ is an ERS, (\ref{eq:PotentialTetali}) implies that we can obtain the electrical potential $\phi^{xy}$ of a unit current between two points $x,y \in V$ via
\begin{equation}
\label{eq:PotentialViaTriangleInequality}
\phi^{xy}(z) = \frac{1}{2}(d(x,y) + d(y,z) - d(x,z)).
\end{equation}
We now turn this statement into a definition for arbitrary finite metric spaces. This gives us the ability to interpret the Dirichlet problem (\ref{eq:DirichletProblem}) as an equation system of the unknown variable $c$ rather than unknown function $\phi$. By grouping together certain equations, we obtain a family of equation systems whose solution, if it exists and satisfies an additional condition, is a graph with effective resistance $d$.
\begin{dfn}
	For $x,y,z \in V$, let
	\begin{equation}
	\label{eq:DefinitionM}
	M_y(x,z) := \frac{1}{2} ( d(x,y) + d(y,z) - d(x,z))
	\end{equation}
	be the defect occurring in the triangle inequality when using the intermediate point $y$. We denote by $M_y'$ the restriction of $M_y$ to the index set $(V \setminus\menge{y}) \times (V \setminus\menge{y})$. Furthermore, let
	\begin{equation}
	\label{eq:DefinitionA}
	A_y(x,z) := \begin{cases}
	1 & , ~ x = y = z\\
	M_y(x,z) & , \text{ otherwise}
	\end{cases}
	\end{equation}
	and
	\begin{equation}
	\label{eq:DefinitionB}
	b_y(x) = 1-\delta_{y}(x).
	\end{equation}
\end{dfn}
\begin{prop}
	\label{prop:CIsSolutionOfLES}
	If $(V,d)$ is an ERS with associated graph $G = (V,c)$ and $y \in V$, then the vector $c(\cdot, y) \in \R^V$ satisfies the equation
	$$A_y \cdot c(\cdot, y) = b_y.$$
\end{prop}
\begin{proof}
	First, note that the equation in the $y$-row simply states that $c(y,y) = 0$ which is true by definition of $c$. Since $(V,d)$ is an ERS, we know that $M_y(x,\cdot) = \phi^{xy}$ is a solution of (\ref{eq:DirichletProblem}). Hence, for $x\neq y$, we have
	$$-\frac{1}{c_y} = (\Delta \phi^{xy})(y) = \phi^{xy}(y) - \sum_{z \in V} \frac{c(y,z)}{c_y} \phi^{xy}(z).$$
	Using $\phi^{xy}(y) = 0$ and the symmetry of $c$, we get 
	$$1 = \sum_{z \in V} c(y,z) \phi^{xy}(z) = \sum_{z \in V} M_y(x,z) c(z,y) = (A_y \cdot c(\cdot, y))(x).	$$
	which concludes the proof.\qed
\end{proof}

It seems natural to wonder about the solvability of (\ref{eq:LES}). For an ERS, it is clear that there exists at least one solution but can there exist more? The following proposition uses the fact that $\Delta$ is a non-negative operator and the strong relationships between $A_y,M_y$ and $\phi^{xy}$ to show that $\det A_y > 0$.

\begin{prop}
	\label{prop:MatrixOfTriangleDefectsHasPositiveDeterminant}
	If $(V,d)$ is an ERS and $y \in V$, we have
	$$\det A_y = \det M_y' > 0.$$
\end{prop}

Before we prove Proposition \ref{prop:MatrixOfTriangleDefectsHasPositiveDeterminant}, let us recall Kirchhoff's matrix tree theorem (see \cite[Theorem VI.29]{Tutte2001}). The Kirchoff matrix of a graph $G = (V, c)$ is defined by
$$K(x,y) = \begin{cases}
c_x & ,  ~ x = y\\
-c(x,y) & , ~ x \neq y
\end{cases}.$$
Then, for any $y \in V$ and $V' := V \setminus\menge{y}$, we have
$$\det(K \restr_{V'}) = \sum_{\substack{T \text{ spanning}\\\text{tree of } G}} \sqrt{ \prod_{\substack{x,z \text{ adjacent}\\\text{ in } T}} c(x,z)} ~ .$$
We observe that $\Delta(x,z) = c_x^{-1} \cdot K(x,z)$ for all $x,z \in V$. Hence,
\begin{equation}
\label{eq:DeterminantOfLaplacianMinorIsPositive}
\det \Delta \restr_{V'} = \det(K\restr_{V'}) \cdot \prod_{x \in V'} \frac{1}{c_x} > 0
\end{equation}
\begin{proof}[Proof of Proposition \ref{prop:MatrixOfTriangleDefectsHasPositiveDeterminant}]
	Fix $y \in V$ and let $V' = V \setminus \menge{y}$. Since $M_y(x,z) = \phi^{xy}(z)$, we have for $x,z \in V'$,
	\begin{align*}
	(\Delta \restr_{V'} \cdot M_y')(x,z) & = \sum_{v \in V'} \Delta(x,v) M_y(v,z) = \sum_{v \in V'} \Delta(x,v) M_y(z,v) \\
	& = \sum_{v \in V'} \Delta(x,v) \phi^{zy}(v) = (\Delta \phi^{zy})(x) = \begin{cases}
	\frac{1}{c_x} & , x = z\\
	0 & , \text{otherwise}
	\end{cases}.
	\end{align*}
	By (\ref{eq:DeterminantOfLaplacianMinorIsPositive}), we have  $\det(\Delta\restr_{V'}) > 0$ and it follows that
	$$\det M_y' = \frac{\det(\Delta\restr_{V'} \cdot M_y')}{\det(\Delta\restr_{V'})} = \det(\Delta\restr_{V'})^{-1} \cdot \prod\limits_{x \neq y} \frac{1}{c_x} >0.$$
	Applying Laplace expansion to the $y$-row of $A_y$ yields $\det A_y = \det M_y'$.\qed
\end{proof}
The combination of Propositions \ref{prop:CIsSolutionOfLES} and \ref{prop:MatrixOfTriangleDefectsHasPositiveDeterminant} gives a necessary condition for a metric space to be an ERS.
\begin{cor}
	\label{cor:NecessaryCondition}
	If a finite metric space $(V,d)$ is an ERS, we have $\det A_y > 0$ for all $y \in V$. Furthermore, the unique solution $(c(x,y))_{x,y \in V}$ of the family of linear equation systems 
	\begin{align*}
	\label{eq:LES}
	A_y \cdot c(\cdot , y) & = b_y ~~,~ y \in V \tag{LES}
	\end{align*}
	has non-negative entries.
\end{cor}
\begin{rem}
	\label{rem:BijectionERSConnectedGraphs}
	Note that $\det A_y > 0$ implies that there exists a \textbf{unique} solution for the linear equation system (\ref{eq:LES}). Hence, for every effective resistance on a finite set $V$ there exists exactly \textbf{one} graph $G = (V,c)$ which induces this effective resistance. In terms of Dirichlet forms and resistance metrics, see Section \ref{section:ResistanceMetrics}, this corresponds to \cite[Theorem 2.1.12]{Kigami2001}. 
\end{rem}
We will use the remainder of this section to prove that the condition stated in Corollary \ref{cor:NecessaryCondition} is not only necessary but also sufficient. From now on we assume that $(V,d)$ is a finite metric space such that $\det A_y > 0$ for all $y \in V$ and that the matrix $(c(x,y))_{x,y \in V}$ satisfying (\ref{eq:LES}) has non-negative entries.

We need to prove two things. First, that $c$ actually defines a graph and second, that the effective resistance of this graph is given by $d$. For the first step, the difficult part is proving that $c$ is symmetric. In order to do so, we will first prove that $\det A_y$ is actually independent of $y$. The essential step is to realize that we can produce $A_y$ from $A_x$ simply by reordering and adding up rows and columns. More precisely, there exist matrices $T_{xy}$ such that $A_y = T_{xy}A_x(T_{xy})^t$.

For $x,y \in V, x \neq y$, let $T_{xy} \in \R^{V \times V}$ be defined by
\begin{equation}
\label{eq:DefinitionT}
T_{xy}(w,z) = \begin{cases} 
1 & , w = z, w \notin \menge{x,y}\\
-1 & , w \neq y, z = y \text{ or } w = y, z = x\\
0 & , \text{ otherwise}
\end{cases}.
\end{equation}

\begin{lem}
	\label{lemma:DeterminantOfT}
	We have $\det T_{xy} = -1$ and $T_{xy}^{-1} = T_{yx}$.
\end{lem}
\begin{proof}
To see that $\det T_{xy} = -1$, add all columns which are not indexed by $x$ or $y$ to the $y$-column and swap the $x$- and $y$-column. The resulting matrix is a diagonal matrix with entries -1 at exactly two positions and 1 otherwise. The fact that $T_{xy}^{-1} = T_{yx}$ is a simple computation.\qed
\end{proof}
\begin{prop}
	\label{prop:TTransformsA}
	For $x,y \in V$, $x \neq y$, we have
	\begin{equation}
	\label{eq:TTransformsA}
	T_{xy} A_x (T_{xy})^t = A_y
	\end{equation}
	and 
	\begin{equation}
	\label{eq:DeterminantOfAIndependent}
	\det M_x' = \det A_x = \det A_y = \det M_y'.
	\end{equation}
\end{prop}
\begin{proof}
For simplicity, let $T = T_{xy}$. Since $(A_x)^t = A_x$, it is clear that $T A_x T^t$ is also symmetric. We now compute
$$(T A_x T^t)(w,z) =  \sum_{a \in V} \sum_{b \in V} T(w,a) \cdot A_x(a,b) \cdot T(z,b) =: \spadesuit$$
for all $w,z \in V$.
\begin{enumerate}[label={Case \arabic*:}, align=left]
	\item $w = z = y$.
	$$\spadesuit = \sum_b T(y,x) \cdot A_x(x,b) \cdot T(y,b) = (-1) \cdot1 \cdot (-1) = 1 = A_y(y,y)$$
	\item $w = z = x$.
	$$\spadesuit = \sum_b T(x,y) A_x(y,b) T(x,b) = (-1) \cdot A_x(y,y) \cdot (-1) = d(x,y) = A_y(x,x)$$
	\item $w = z, w \notin \menge{x,y}$:
	\begin{align*}
	\spadesuit & = \sum_b T(w,w) A_x(w,b)T(w,b) + \sum_b T(w,y) A_x(y,b) T(w,b)\\
	& = A_x(w,w) - A_x(w,y) - A_x(y,w) + A_x(y,y)\\
	& = d(x,w) - (d(w,x) + d(x,y) - d(w,y)) + d(x,y) = d(w,y) = A_y(w,w)
	\end{align*}
	\item $w = y, z \neq y$.
	\begin{align*}
	\spadesuit & = \sum_b T(y,x) A_x(x,b) T(z,b) = (-1) (A_x(x,y) T(z,y) + A_x(x,z) T(z,z))\\
	& = - \underbrace{A_x(x,z)}_{= \delta_{xz}} \underbrace{T(z,z)}_{= 1- \delta_{xz}} = 0 = A_y(y,z)
	\end{align*}
	\item $w = x, z \notin \menge{x,y}$.
	\begin{align*}
	\spadesuit & = \sum_b T(x,y) A_x(y,b) T(z,b) = - (A_x(y,z) - A_x(y,y))\\
	& = d(x,y) - \frac{1}{2}(d(x,y) + d(x,z) - d(y,z)) = \frac{1}{2}( d(x,y) + d(y,z) - d(x,z)) \\
	& = A_y(x,z)
	\end{align*}
	\item $w \neq z$ and $w,z \notin \menge{x,y}$.
	\begin{align*}
	\spadesuit & = \sum_b A_x(w,b) T(z,b) - \sum_b A_x(y,b) T(z,b)\\
	& = (A_x(w,z) - A_x(w,y)) - (A_x(y,z) - A_x(y,y))\\
	& = \frac{1}{2}\ls d(x,z) - d(x,y) - d(w,z) + d(w,y) - d(x,y) - d(x,z) + d(y,z)\rs + d(x,y)\\
	& = \frac{1}{2}(d(w,y) + d(y,z) - d(w,z)) = A_y(w,z)
	\end{align*}
\end{enumerate}
This concludes the proof of (\ref{eq:TTransformsA}). By Lemma \ref{lemma:DeterminantOfT}, we have
$$\det A_y = \det( T_{xy}A_x (T_{xy})^t ) = \det T_{xy} \det A_x \det T_{xy} = (-1) \det A_x (-1)  = \det A_x.$$
This implies (\ref{eq:DeterminantOfAIndependent}) by applying Laplace expansion to obtain $\det A_y = \det M_y'$.\qed
\end{proof}
In what follows, we denote by $A_{xy}$ the matrix that results from replacing the $y$-column in $A_x$ with $b_x$. Then, Cramer's rule applied to the linear equation system $A_x \cdot c(\cdot, x) = b_x$ yields
$$c(y,x) = \frac{\det A_{xy}}{\det A_x}.$$
From Proposition \ref{prop:TTransformsA}, we already know how the transformation $A \mapsto T_{xy} A (T_{xy})^t$ acts on the matrix $A_x$. We will now investigate its effects on $A_{xy}$.
\begin{lem}
	\label{lemma:Laplace}
	Let $x,y \in V$, $x \neq y$ and $A := T_{xy} A_{xy} T_{xy}^t$. Then,
	$$A(w,x) = \delta_{wx} \text{  and  } A(w,y) = \delta_{wy}$$
	holds for all $w \in V$.
\end{lem}

\begin{proof}
	For simplicity, let $T = T_{xy}$. For $z = x$ or $z = y$ and $b \in V$, we have
	$$T(z,b) = \begin{cases}
	-1 & , z = x, b = y \text{ or } z= y, b = x\\
	0 & , \text{ otherwise}
	\end{cases}$$
	We consider three cases and compute
	$$A(w,z) = \sum_{a \in V} \sum_{b \in V} T(w,a) A_{xy}(a,b) T(z,b).$$
	\begin{enumerate}[label={Case \arabic*:}, align=left]
		\item $w = x$.
		\begin{align*}
		A(x,x) & = A_{xy}(y,y) = b_x(y) = 1\\
		A(x,y) & = A_{xy}(y,x) = 0
		\end{align*}
		\item $w = y$.
		\begin{align*}
		A(y,x) & = A_{xy}(x,y) = b_x(x)  = 0\\
		A(y,y) & = A_{xy}(x,x) = 1
		\end{align*}
		\item $w \notin \menge{x,y}$.
		\begin{align*}
		A(w,x)& = - A_{xy}(w,y) + A_{xy}(y,y) = b_x(y) - b_x(w) = 0\\
		A(w,y) & = -A_{xy}(w,x) + A_{xy}(y,x) = 0
		\end{align*}
	\end{enumerate}
	This concludes the proof.\qed
\end{proof}

\begin{prop}
	\label{prop:detAdetB}
	Let $x,y \in V$, $x \neq y$, $A := T_{xy} A_{xy} T_{xy}^t$ and $B := T_{yx} A_{yx} T_{yx}^t$. Then, $A(w,z) = B(z,w)$ for all $w,z \notin \menge{x,y}$ and it follows that
	$$\det A  = \det B.$$
\end{prop}

\begin{proof}
	We have
	\begin{align*}
	A(w,z) & = \sum_{a \in V} \sum_{b \in V} T_{xy}(w,a) A_{xy}(a,b) T_{xy}(z,b) \\
	& = \sum_b A_{xy}(w,b) T(z,b) - A_{xy}(y,b) T(z,b)\\
	& = A_{xy}(w,z) - \underbrace{A_{xy}(w,y)}_{=1} - (A_{xy}(y,z) - \underbrace{A_{xy}(y,y)}_{=1})\\
	& = \frac{1}{2}(d(x,w) + d(y,z) - d(w,z) - d(y,x))\\
	& = A_{yx}(z,w) - A_{yx}(x,w) - (\underbrace{A_{yx}(z,x)}_{=1} - \underbrace{A_{yx}(x,x)}_{=1})\\
	& = \sum_b A_{yx}(z,b) T_{yx}(w,b) - A_{yx}(x,b) T_{yx}(w,b)\\
	& = \sum_a\sum_b T_{yx}(z,a) A_{yx}(a,b) T_{yx}(w,b) = B(z,w).
	\end{align*}
	Hence, we compute
	$$\det A = 1 \cdot 1 \cdot \det\lb (A(w,z))_{w,z \in V \setminus \menge{x,y}} \rb = 1 \cdot 1 \cdot \det\lb B(z,w)_{z,w \in V \setminus \menge{x,y}} \rb = \det B$$
	by applying Lemma \ref{lemma:Laplace} and Laplace expansion.	\qed
\end{proof}
We can now combine these computations to obtain the symmetry of the solution matrix $c$ of (\ref{eq:LES}).
\begin{prop}
	\label{prop:Symmetry}
	For $x,y \in V$, we have $c(x,y) = c(y,x)$.
\end{prop}
\begin{proof}
	For $x = y$, there is nothing to show. For $x \neq y$, it follows from Proposition \ref{prop:detAdetB} that
	$$\det A_{xy} = \det( T_{xy} A_{xy} T_{xy}^t) = \det( T_{yx} A_{yx} T_{yx}^t) = \det A_{yx}.$$
	Hence, 
	\begin{equation*}
	c(x,y) = \frac{\det A_{yx}}{\det A_y} = \frac{\det A_{xy}}{\det A_x} = c(y,x). \tag*{\qed}
	\end{equation*}
\end{proof}

Since $c$ is a symmetric matrix with a vanishing diagonal and non-negative entries, it defines a graph $G = (V,c)$. At this point we do not know whether this graph is connected. We do know, however, that it has no isolated vertices since this would correspond to a column in $c$ consisting only of zeros. This is clearly impossible due to the form of (\ref{eq:LES}). Since there are no isolated vertices, the Laplacian of $G$  is well-defined. By Lemma \ref{lemma:PotentialIffGraphConnected}, ($\ref{eq:DirichletProblem}$) has a solution for all $x,y \in V$ if and only if $G$ is connected and we will show that $\phi^{xy}(z) := M_y(x,z)$ is such a solution for the Laplacian of $G$, implying that $G$ is connected.

\begin{prop}
	\label{prop:CAndMyAreCompatible}
	For $x,y \in V$, $x \neq y$, let $\phi^{xy}(z) := M_y(x,z)$. Then,
	\begin{align*}
	\Delta \phi^{xy} & = \frac{1}{c_x} \mathds{1}_x - \frac{1}{c_y} \mathds{1}_y\\
	\phi^{xy}(y) & = 0 ~.
	\end{align*}
	In particular, $d(x,y)$ is the effective resistance between $x$ and $y$ in the graph $(V,c)$.
\end{prop}
\begin{proof}
First, observe some basic properties of $M_y$. For $w,x,z \in V$, we have
\begin{align}
M_y(x,x) & = d(x,y) = M_x(y,y)\label{eq:PropM1}\\
M_y(x,z) & = M_y(z,x)\label{eq:PropM2}\\
M_y(x,z) & = d(x,y) - M_x(y,z)\label{eq:PropM3}\\
M_y(x,w) - M_y(x,z) & = M_w(y,z) - M_w(x,z)\label{eq:PropM4}\\
& = M_z(x,w) - M_z(y,w).\notag
\end{align}
By (\ref{eq:LES}), we have $c(x,x) = 0$ for all $x \in V$ and
\begin{align*}
1 & = (A_y \cdot c(\cdot, y))(x) = \sum_{w \in V} M_y(x,w) c(w,y)
\end{align*}
for all $x,y \in V$ such that $x \neq y$. Using (\ref{eq:PropM1}) and (\ref{eq:PropM3}), we compute
\begin{align*}
c_x \cdot \Delta \phi^{xy}(x) & = c_x \phi^{xy}(x) - \sum_w c(x,w) \phi^{xy}(w) = c_x M_y(x,x) - \sum_w c(w,x) M_y(x,w)\\
& = c_x d(x,y) - \sum_w c(w,x) (d(x,y) - M_x(y,w)) = \sum_w  M_x(y,w) c(w,x)\\
& = (A_x \cdot c(\cdot, x))(y) = 1.
\end{align*}
For $z \notin \menge{x,y}$, we have
$$\sum_w M_z(x,w) c(w,z) = (A_z \cdot c(\cdot, z))(x) = 1 = (A_z \cdot c(\cdot, z))(y) = \sum_w M_z(y,w) c(w,z).$$
Combining this with (\ref{eq:PropM4}), we get
\begin{align*}
c_z (\Delta \phi^{xy})(z) & = c_z \phi^{xy}(z) - \sum_w c(z,w) \phi^{xy}(w) = c_z M_y(x,z) - \sum_w c(w,z) M_y(x,w)\\
& = \sum_w c(w,z) [M_y(x,z) - M_y(x,w)] =  \sum_w c(w,z) [M_z(y,w) - M_z(x,w)]\\
& = 1- 1 = 0. 
\end{align*}
Lastly, $M_y(x,y) = 0$ yields
\begin{align*}
c_y (\Delta \phi^{xy})(y) & = c_y \phi^{xy}(y) - \sum_{w} c(y,w) \phi^{xy}(w) = c_y M_y(x,y) - \sum_w c(w,y) M_y(x,w)\\
& = 0 - 1 = -1.
\end{align*}
Now that we have established that $\phi^{xy} = M_y(x,\cdot)$ solves the Dirichlet problem (\ref{eq:DirichletProblem}) on $G$, it is a direct consequence that the effective resistance $R$ of $G$ is in fact the metric $d$.
\begin{equation*}
R(x,y) = \phi^{xy}(x) = M_y(x,x) = d(x,y)\tag*{\qed}
\end{equation*}
\end{proof}

Combining all results of this section, we obtain a concise characterization of finite effective resistance spaces.

\begin{theorem}
	\label{theorem:CharacterizationFiniteERS}
	Let $(V,d)$ be a finite metric space and let $M_y',A_y,b_y$ be as in (\ref{eq:DefinitionM}) through (\ref{eq:DefinitionB}). Then, there exists a graph $G = (V,c)$ with effective resistance $d$ if and only if
	$$\det M_y' > 0$$
	for some $y \in V$ and the family of linear equation systems
	$$A_y \cdot c(\cdot , y) = b_y ~~ , ~ y \in V$$
	admits a solution matrix $(c(x,y))_{x,y \in V}$ with non-negative entries.
\end{theorem}

\begin{exa}[Geodesic metric of cycle graph]
	We give an example for a naturally occurring metric space which is not representable as an effective resistance. Let $V = \menge{v_0,\ldots, v_3}$ and consider the graph $\mathcal{C}_4 = (V, c)$ where 
$$c(v_i, v_j) = \begin{cases}
1 & , ~ |i-j| \in \menge{1,3}\\
0 & , ~ \text{ otherwise}
\end{cases},$$
see Figure \ref{fig:CycleGraph4}. 
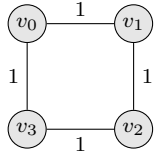
\begin{figure}
	\centering
	\begin{tikzpicture}[-,auto, node distance = 1.4cm, every loop/.style={}, vertex/.style={draw, circle, fill=black!10, inner sep=0mm, minimum size=5mm}]
	\node[vertex] (A) {$v_0$};
	\node[vertex, right of= A] (B) {$v_1$};
	\node[vertex,below of =B] (C) {$v_2$};
	\node[vertex,left of=C] (D) {$v_3$};
	
	\path 	(A) edge node {$1$} (B)
	(B) edge node {$1$} (C)
	(C) edge node {$1$} (D)
	(D) edge node {$1$} (A);
	\end{tikzpicture}
	\caption{\label{fig:CycleGraph4}The cycle-graph $\mathcal{C}_4$.}
\end{figure}
The \emph{geodesic metric} $d_G(x,y)$ between vertices $x,y \in V$ of a weighted graph $G = (V',c')$ is defined as the length of the shortest path from $x$ to $y$, i.e.,
$$d_G(x,y) = \inf\mengedef{\sum_{k=1}^{n-1} c'(x_k,x_{k+1})^{-1}}{(x_1,\ldots, x_n) \text{ path from } x \text{ to } y \text{ in } G}.$$
For $\mathcal{C}_4$, we have
$$(d_{\mathcal{C}_4}(v_i,v_j))_{i,j\in \menge{0,\ldots,3}} = \begin{pmatrix}
0 & 1 & 2 & 1\\
1 & 0 & 1 & 2\\
2 & 1 & 0 & 1\\
1 & 2 & 1 & 0 
\end{pmatrix} \text{ and } M_{v_0}' = \begin{pmatrix}
1 & 1 & 0\\
1 & 2 & 1\\
0 & 1 & 1
\end{pmatrix}.$$
Hence, $\det M_{v_0}' = 0$ and thus $(V,d_{\mathcal{C}_4})$ is no effective resistance space.
\end{exa}
\begin{exa}
	This example demonstrates that the second condition in Theorem \ref{theorem:CharacterizationFiniteERS}, namely that the family of linear equation systems admits a solution with non-negative entries, can not be dropped. Moreover, it does not suffice to check ${c(\cdot,y) \geq 0}$ for only one $y \in V$.
	Consider the metric space $(V, d)$ where $V = \menge{v_0,v_1,v_2,v_3}$ and $d$ is given by
	$$ (d(v_i,v_j))_{i,j \in \menge{0,\ldots,3}} = \frac{1}{260}
	\begin{pmatrix}
	0 & 23 & 36 & 40 \\ 
	23 & 0 & 39 & 23 \\ 
	36 & 39 & 0 & 36 \\ 
	40 & 23 & 36 & 0
	\end{pmatrix}.$$
	It follows that 
	$$M_{v_0}' = \frac{1}{130}\begin{pmatrix}
	23 & 10 & 20\\
	10 & 36 & 20\\
	20 & 20 & 40  
	\end{pmatrix}$$ and thus $\det M_{v_0}' = \frac{26}{4225} > 0$. Hence, the family of linear equation systems (\ref{eq:LES}) has a unique solution matrix $c$. This matrix has two negative entries. Interpreting it as a graph with negative edge weights, we get $G_- = (V,c)$ which is shown in Figure \ref{fig:GraphOneNegativeEdge}.
	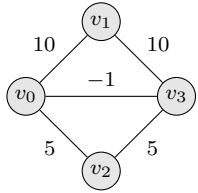
\begin{figure}
		\centering
		\begin{tikzpicture}[-,auto, node distance = 1.4cm, every loop/.style={}, vertex/.style={draw, circle, fill=black!10, inner sep=0mm, minimum size=5mm}]
		\node[vertex] (A) {$v_0$};
		\node[vertex, above right of= A] (B) {$v_1$};
		\node[vertex, below right of =A] (C) {$v_2$};
		\node[vertex, below right of=B] (D) {$v_3$};
		
		\path 	
		(A) edge node {$10$} (B)
		(A) edge node [below left] {$5$} (C)
		(A) edge node {$-1$} (D)
		(B) edge node {$10$} (D)
		(C) edge node [below right] {$5$} (D);
		\end{tikzpicture}
		\caption{\label{fig:GraphOneNegativeEdge}The graph $G_-$}
	\end{figure}
	We see that $c(v_0,v_3) < 0$ and this is also the only negative entry of $c$ (apart from $c(v_3,v_0)$).
\end{exa}

\begin{rem}
	In the example above, the energy form $\mathcal{E}$ of $G_-$ satisfies \ref{DF1} and \ref{DF2} but lacks the Markov property (cf. Definition \ref{dfn:DirichletForm}).
\end{rem}

\section{Resistance metrics}
\label{section:ResistanceMetrics}
We aim to apply Theorem \ref{theorem:CharacterizationFiniteERS} to the theory of \emph{resistance metrics} introduced by Kigami in \cite{Kigami2001}. For the convenience of the reader, we recall some definitions and results.

\begin{dfn}[Kigami]
	\label{dfn:DirichletForm}
	Let $V$ be a finite set. We denote by $l(V)$ the set of all functions $u: V \to\R$. A symmetric quadratic form $\F : l(V) \to \R$ is called a \emph{Dirichlet form} if the following conditions are met.
	\begin{enumerate}[label=(DF-\arabic*), align=left]
		\item \label{DF1}$\F(u) \geq 0$ for all $u \in l(V)$.
		\item \label{DF2}$\F(u) = 0$ if and only if $u$ is constant on $V$.
		\item \label{DF3}For $u \in l(V)$ we have $\F(\tilde{u}) \leq \F(u)$ where
		$$\tilde{u}(x) = \min\{1,\max\{u(x),0\}\}.$$
	\end{enumerate}
	Condition \ref{DF3} called the \emph{Markov property}.
\end{dfn}
A very accessible example for Dirichlet forms are \emph{energy forms} of connected graphs. For such a graph $G = (V,c)$, its energy form $\mathcal{E}_G$ is defined by
\begin{equation}
\mathcal{E}_G(u) := \frac{1}{2}\sum_{x,y \in V} c(x,y) (u(x)-u(y))^2 ~ , ~ u \in l(V).
\end{equation}
The following Lemma states the folkloric result that energy forms of connected graphs are in fact the only examples of Dirichlet forms on finite sets. We include it for future reference. A proof can be found in \cite{Kigami2001}.
\begin{lem}
	\label{lemma:DirichletFormsEqualEnergyFormsOfGraphs}
	Let $V$ be a finite set, $|V|\geq 2$. A symmetric quadratic form $\F : l(V) \to \R$ is a Dirichlet form on $V$ if and only if there exists a connected graph $G = (V,c)$ such that $\F$ is its energy form. In particular, for $x,y \in V$, $x \neq y$, we have
	\begin{equation}
	\label{eq:EdgeWeightsFromEnergyForm}
	c(x,y) = - \F(\mathds{1}_x, \mathds{1}_y).
	\end{equation}
\end{lem}
For a Dirichlet form $\F$ on a finite set $V$, define
\begin{equation}
R_{\F}(x,y) =\left( \min \mengedef{\F(u)}{u \in l(V), u(x) = 1, u(y) = 0} \right)^{-1}.
\end{equation}
Let $G$ be a finite graph, $\mathcal{E}$ its energy form, $R$ its effective resistance and $\phi^{xy}$ the solution of the Dirichlet problem (\ref{eq:DirichletProblem}) associated to $G$. By \cite[Theorem 2.3]{JorgensenPearse2009}, we have the equality
\begin{equation}
\label{eq:EffectiveResistanceViaDirichletForm}
R_{\F}(x,y) = \F(\phi^{xy}) = R(x,y).
\end{equation}

By \cite[Theorem 2.1.12]{Kigami2001}, two Dirichlet forms $\F_1$ and $\F_2$ on $V$ satisfy $\F_1 = \F_2$ if and only if $R_{\F_1} = R_{\F_2}$. Together with (\ref{eq:EffectiveResistanceViaDirichletForm}), this implies that there is a bijection between all Dirichlet forms on $V$ and all effective resistance spaces on $V$.

\begin{dfn}[Kigami]
	Let $X$ be any set. A function $R: X \times X \to \R_{\geq 0}$ is called \emph{resistance metric} on $X$ if, for any finite $W \subseteq X$, there exists a Dirichlet form $\F_W$ such that
	$$R\restr_{W \times W} = R_{\F_W}.$$
\end{dfn}
In order to apply our characterization from Theorem \ref{theorem:CharacterizationFiniteERS} to resistance metrics, we need the notion of the \emph{trace} of a Dirichlet form. For a Dirichlet form $\F$ on a set $X$ and a finite subset $W\subseteq X$, \cite{Kasue2010} defines the trace $\F^W$ of $\F$ by
$$\F^W(w) := \min_{\substack{u: X \to \R\\u\restr_W = w}} \F(u) ~,~ w \in l(W).$$ 
By Lemma 1.7 and Lemma 1.9 in \cite{Kasue2010}, $\F^W$ is then a Dirichlet form on $W$ satisfying
$$R_{\F}(x,y) = R_{\F^W}(x,y) \all x,y \in W.$$
Using the trace of Dirichlet forms, we see that for a finite set $X$, $R$ is a resistance metric on $X$ if and only if there exists a graph $G = (X,c)$ with effective resistance $R$.
\begin{theorem}
	\label{theorem:CharResistanceMetrics}
	Let V be a countably infinite set $V$ and $(V_n)_{n\in\N}$ a finite exhaustion of $V$. Then, a function $R: V \times V \to \R_{\geq 0}$ is a resistance metric on $V$ if and only if $R_n := R \restr_{V_n \times V_n}$ satisfies the criterion of Theorem \ref{theorem:CharacterizationFiniteERS} for all $n \in \N$.
\end{theorem}
\begin{proof}
	Let $R$ be a resistance metric and $n \in \N$. By definition, $R_n = R_{\F_n}$ holds for some Dirichlet form $\F_n$ on $V_n$. By Lemma \ref{lemma:DirichletFormsEqualEnergyFormsOfGraphs} and (\ref{eq:EffectiveResistanceViaDirichletForm}), there exists a finite graph $G_n = (V_n, c_n)$ such that $R_{\F_n} = R_{G_n}$. Hence, $R_n$ is a finite ERS and thus satisfies the criterion of Theorem \ref{theorem:CharacterizationFiniteERS}.
	
	Now assume that $R_n = R\restr_{V_n\times V_n}$ satisfies the criterion of Theorem \ref{theorem:CharacterizationFiniteERS} for all $n \in \N$. Hence, $R_n$ is the effective resistance of a finite graph $G_n = (V_n, c_n)$ with energy form $\mathcal{E}_n$. Let $W \subseteq V$ be any finite subset of $V$. Since $\bigcup V_n = V$, there exists $n \in \N$ such that $W \subseteq V_n$. Then, $(\mathcal{E}_n)^W$ is a Dirichlet form on $W$ and satisfies
	$$R_{(\mathcal{E}_n)^W}(x,y) = R_{\mathcal{E}_n}(x,y) = R_n(x,y) = R(x,y)$$
	for all $x,y \in W$. It follows that $R$ is a resistance metric.\qed
\end{proof}

\section{The limit graph of countably infinite resistance metrics}
\label{section:LimitingProperties}
Let $V$ be a countably infinite set and $R$ be a resistance metric on $V$. If $(V_n)_{n \in \N}$ is a finite exhaustion of $V$, each $(V_n, R_n)$ with $R_n = R \restr_{V_n}$ is an effective resistance space. Hence, there exists $G_n = (V_n, c_n)$ with effective resistance $R_n$. For $x,y \in V$, we are now interested in the limiting behavior of the sequence $(c_n(x,y))_{n \in \N}$.

First, recall a well-known method of network reduction called the \emph{star-mesh transform} \cite[Exercise 2.69(e)]{LyonsPeres2015}. It is used to remove a node from an electrical network without affecting the effective resistance between all other nodes by possibly adding new resistors. In our context of weighted graphs, it can be formulated as follows.
\begin{lem}[Star-Mesh-Transform]
\label{lemma:StarMesh}
Let $(V,c)$ be a finite graph with effective resistance $R$ and fix $x_0 \in V$. Furthermore, let $R'$ denote the effective resistance of the network $(V \setminus \menge{x_0}, c')$ where
\begin{equation}
\label{eq:StarMeshEdgeWeights}
c'(x,y) := \begin{cases}
0 & , x = y\\
c(x,y) + \frac{c(x,x_0)c(x_0,y)}{c_{x_0}} & , x \neq y
\end{cases}~.
\end{equation}
Then, $R'(x,y) = R(x,y)$ for all $x,y \in V \setminus \menge{x_0}$.
\end{lem}
By Remark \ref{rem:BijectionERSConnectedGraphs}, the graph defined by (\ref{eq:StarMeshEdgeWeights}) is the \textbf{only} graph with vertex set $V \setminus \menge{x_0}$ inducing this effective resistance. Hence, the star-mesh transform is the only way to remove a vertex without changing the effective resistances between all other vertices. Observe that (\ref{eq:StarMeshEdgeWeights}) also implies an explicit formula for the sum of edge weights at a given vertex $x$, namely
\begin{align}
\label{eq:StarMeshVertexWeights}
c'_x & = c_x - c(x,x_0) + \sum_{w \neq x, x_0} \frac{c(x,x_0)c(x_0,w)}{c_{x_0}} \\
& = c_x + c(x,x_0) \lb -1 + \frac{c_{x_0} - c(x_0,x)}{c_{x_0}} \rb = c_x - \frac{c(x_0,x)^2}{c_{x_0}}~.\notag
\end{align}
\begin{rem}[Probabilistic interpretation]
	Let $\P_x$ and $\P'_x$ denote the random walk on $(V,c)$ and $(V\setminus\menge{x_0},c')$, respectively, and let $y \in V\setminus\menge{x_0,x}$. Then,
	\begin{align*}
	\P'_x[\omega_1 = y] & = \frac{c'(x,y)}{c'_x} = \frac{c_x}{c'_x} \left(\frac{c(x,y)}{c_x} + \frac{c(x,x_0)c(x_0,y)}{c_xc_{x_0}}\right)\\
	& = \left(1 - \frac{c(x,x_0)c(x_0,x)}{c_xc_{x_0}}\right)^{-1} \left( \P_x[\omega_1 = y] + \P_x[\omega_1 = x_0, \omega_2 = y]\right)\\
	& = \left( \P_x[(\omega_1, \omega_2) \neq (x_0,x)] \right)^{-1} \cdot \left( \P_x[\omega_1 = y] + \P_x[\omega_1 = x_0, \omega_2 = y]\right)\\
	& = \P_x[\omega_1 = y \vee \omega_1 = x_0, \omega_2 = y ~ | ~ (\omega_1, \omega_2) \neq (x_0,x)].
	\end{align*}
	The behavior of $\P_x'$ is very similar to that of $\P_x$ with only two differences. First, the removal of $x_0$ from $(V,c)$ is compensated by adding a \emph{shortcut} from $x$ to $y$, enabling $\P'_x$ to jump directly from $x$ to $y$ whenever $\P_x$ would have taken a detour over $x_0$. Second, the transition probabilities of $\P_x'$ are those of $\P_x$ conditioned on $(\omega_1,\omega_2) \neq (x_0,x)$. This is due to the fact that our version of the star-mesh transform stays within our class of weighted graphs; more precisely, graphs which do not admit self-loops. Hence, the shorted random walk forgets about the possibility of going to $x_0$ and back to $x$ again.
\end{rem}
\begin{rem}[Trace of Energy form]
	Let $\mathcal{E}$ be the energy form of $(V,c)$ and $W := V \setminus\menge{x_0}$. Then, the trace $\mathcal{E}^{V \setminus\menge{x_0}}$ of $\mathcal{E}$ is exactly the energy form of $(V \setminus\menge{x_0}, c')$.
\end{rem}
\begin{cor}
\label{cor:MonotonyOfEdgeWeights}
Let $V_1, V_2$ be two finite sets such that $V_1 \subseteq V_2$ and let $(V_1,c_1)$ and $(V_2,c_2)$ be two graphs with effective resistances $R_1$ and $R_2$. If $R_1(x,y) = R_2(x,y)$ for all $x,y \in V_1$, then
$$c_1(x,y) \geq c_2(x,y) ~ \all x,y \in V_1$$
and
$$(c_1)_x \leq (c_2)_x \all x \in V_1.$$
In particular,
$$\frac{c_1(x,y)}{(c_1)_x} \geq \frac{c_2(x,y)}{(c_2)_x} \all x,y \in V_1.$$
\end{cor}
\begin{proof}
	The claim follows from (\ref{eq:StarMeshEdgeWeights}) and (\ref{eq:StarMeshVertexWeights}) by successively applying the star-mesh transform to each vertex in $V_2 \setminus V_1$ .\qed
\end{proof}
\begin{theorem}
	\label{theorem:KantenLimesEx}
	Let $R$ be a resistance metric on a countably infinite set $V$ and let $(V_n)_{n \in \N}$ be a finite exhaustion of $V$. Denote by $G_n = (V_n, c_n)$ the unique graph with effective resistance $R(x,y)$ for all $x,y \in V_n$. Then,
	$$\lim_{n \to \infty} c_n(x,y) \in [0,+\infty)$$
	exists for all $x,y \in V$ and is independent of the choice of $(V_n)_{n \in \mathbb N}$. Furthermore,
	$$\lim_{n \to \infty} (c_n)_x \in (0,+\infty]$$
	exists for all $x \in V$ and is independent of the choice of $(V_n)_{n \in \mathbb N}$.
\end{theorem}

\begin{proof}[Proof of Theorem \ref{theorem:KantenLimesEx}]
Let $x,y \in V_m$. By Corollary \ref{cor:MonotonyOfEdgeWeights}, we have
$$c_n(x,y) \geq c_{n+1}(x,y) \all n \geq m$$
Since $c_n(x,y)\geq 0$ for all $n \geq m$, we see that $\lim_{n \to\infty} c_n(x,y)$ does indeed exist. Its independence regarding the choice of $(V_n)_{n\in\N}$ is a generic argument. For any two exhaustions $\mathcal{V} = (V_n)_{n\in\N}$ and $\mathcal{V}' = (V_n')_{n\in \N}$ with limits $c(x,y)$ and $c'(x,y)$, there exists an exhaustion which contains infinitely many members of $\mathcal{V}$ as well as of $\mathcal{V}'$. Since the limit exists for this new exhaustion, $c(x,y) = c'(x,y)$ follows.

The second statement follows analogously since Corollary \ref{cor:MonotonyOfEdgeWeights} also states that $(c_n)_x \leq (c_{n+1})_x$ for all $n \geq m$. The only difference is that we do not have an upper bound for the sequence and hence it may be possible that $(c_n)_x \to \infty$. \qed
\end{proof}
\begin{dfn}
	\label{definition:LimitGraph}
	Using the notation of Theorem \ref{theorem:KantenLimesEx}, we define the \emph{limit graph} of $R$ to be $G_R := (V,c)$ with
	$$c(x,y) := \lim_{n \to \infty} c_n(x,y) ~ , ~ x,y \in V.$$
\end{dfn}
\begin{rem}
	The definition of $G_R$ does not depend on the choice of the exhaustion $(V_n)_{n\in\N}$.
	 
	Furthermore, $G_R$ is a weighted graph as defined in Section \ref{section:Introduction}. Indeed, for $x,y \in V$, we have $c(x,x) = \lim_{n\to \infty} c_n(x,x) = 0$ and $c(x,y) = \lim_{n \to\infty} c_n(x,y) = \lim_{n \to \infty} c_n(y,x) = c(y,x)$.
\end{rem}
It seems natural to wonder about the connection between $G_R$ and the metric $R$. One would hope that $R$ is the effective resistance of $G_R$. On infinite graphs there are two prominent versions of effective resistance called \emph{free} and \emph{wired} effective resistance. See \cite[Section 2]{JorgensenPearse2009} or \cite[Section 9]{LyonsPeres2015} for definitions. Both notions require the infinite graph to satisfy $c_x < \infty$ for every $x \in V$ which unfortunately is not true in general for $G_R$. Indeed, the following examples show that $G_R$ may not allow for a well-defined notion of its Laplacian.

\begin{exa}[Discrete metric on $\N$]
\label{ex:LimitGraphDegenerate}	
Consider $V = \N$ and the discrete metric $d_{\text{disc}}$ on $V$, i.e. $d_{\text{disc}}(x,y) = 1$ for all $x \neq y$. Now take any finite exhaustion $(V_n)_{n\in\N}$ of $V$. Then, $(V_n, d_{\text{disc}} \restr_{V_n})$ is an ERS with associated graph $(V_n, c_n)$ where
$$c_n(x,y) = (1-\delta_{xy}) \frac{2}{|V_n|}.$$
It follows that $c(x,y) = 0$ for all $x,y \in V$ since $|V_n| \to \infty$. Thus, $G_{d_{\text{disc}}}$ is completely disconnected and its Laplacian is not well-defined.
\end{exa}

\begin{exa}[Star metric on $\N$]
	\label{ex:LimitGraphNoLaplacian}	
	Consider $V = \N$ equipped with the star metric $d_{\bigstar}$, i.e.,
	$$d_{\bigstar}(x,y) = \begin{cases}
	0 & , ~ x = y\\
	1 & , ~ x = 1 \text{ or } y = 1\\
	2 & , ~ x \neq 1 \neq y
	\end{cases}.$$
	Furthermore, consider the finite exhaustion $(V_n)_{n \in \N}$, $V_n = \menge{1,\ldots, n}$, of $V$. Then, $d_{\bigstar} \restr_{V_n}$ is the effective resistance of $(V_n, c_n)$ where 
	$$c_n(x,y) = \begin{cases}
	1 & , ~ x = 1, y \neq 1 \text{ or } x \neq 1, y = 1\\
	0 & , \text{ otherwise}
	\end{cases}.$$
	This is due to the fact that $(V_n, c_n)$ is a tree (it contains no cycles) and on a tree the effective resistance equals the geodesic metric, see \cite[Lemma 4.3]{JorgensenPearse2009}. Hence, $G_{d_{\bigstar}}$ is connected but since $c_1 = \sum_{n \in \N} c(1,n) = \infty$, its Laplacian is not well-defined.	
\end{exa}

If the resistance metric $R$ is induced by a graph $G = (V,c)$, the following proposition shows that the limit graph of $R$ is exactly $G$. To formalize what it means for $R$ to be \emph{induced by $G$}, we will use the notion of a \emph{resistance form}. However, as the precise definition is not needed for the statement itself, we will omit it and refer the interested reader to \cite[Definition 2.3.1]{Kigami2001} or \cite[Definition 1.2]{Kasue2010}. 

For a graph $G = (V,c)$ with energy form $\mathcal{E}$, we define $\domE$ and $\Fin$ similar to \cite{JorgensenPearse2009}, namely
\begin{equation}
\dom \mathcal{E} := \mengedef{u: V \to \R}{\mathcal{E}(u) < \infty}
\end{equation}
and
\begin{equation}
\Fin := \mengedef{u: V \to \R}{ \ex k \in \R, W \subseteq V : |W| < \infty \text{ and } u\restr_{V \setminus W} \equiv k}.
\end{equation}
\begin{prop}
	\label{prop:LimitGraphConsistency}
	Let $G = (V,c)$ be a countably infinite graph with energy form $\mathcal{E}$. Furthermore, let $\Fin \subseteq \mathcal{D} \subseteq \domE$ such that $(\mathcal{E}, \mathcal{D})$ is a resistance form with corresponding resistance metric $R$, i.e.
	\begin{equation}
	R(x,y) = \lb\inf \mengedef{\mathcal{E}(u)}{u(x) = 1, u(y) = 0, u \in \mathcal{D}}\rb^{-1}.
	\end{equation}
	Then, $G$ is the limit graph of $R$.
\end{prop}
\begin{proof}
	For $x,y \in V$, $x \neq y$, we have $c(x,y) = -\mathcal{E}(\mathds{1}_x, \mathds{1}_y)$.	Let $(V_n)_{n \in \N}$ be a finite exhaustion of $V$ and $G_n = (V_n, c_n)$ be the unique graph such that $R_{G_n} = R \restr_{V_n}$. 
	By Lemma \ref{lemma:DirichletFormsEqualEnergyFormsOfGraphs} and the uniqueness of $G_n$, it follows that $\mathcal{E}_{G_n}$ is the trace of $(\mathcal{E}, \mathcal{D})$ on $V_n$, i.e., $\mathcal{E}_{G_n} = \mathcal{E}^{V_n}$. By \cite[Lemma 1.10]{Kasue2010}, we have for $u,v \in \mathcal{D}$,
	$$\mathcal{E}(u,v) = \lim_{n \to \infty} \mathcal{E}^{V_n}(u\restr_{V_n}, v\restr_{V_n}).$$
	Hence, for $x,y \in V$, $x \neq y$, we have
	\begin{align*}
	c(x,y) & = - \mathcal{E}(\mathds{1}_x, \mathds{1}_y) = -\lim_{n \to \infty} \left( \mathcal{E}^{V_n}(\mathds{1}_x, \mathds{1}_y)\right) \\
	& = -\lim_{n \to \infty} \left(\mathcal{E}_{G_n}(\mathds{1}_x, \mathds{1}_y)\right) = \lim_{n\to \infty} c_n(x,y) = c_R(x,y)
	\end{align*}
	It follows that $c = c_R$ and therefore $G = G_R$.\qed
\end{proof}
\begin{rem}
	Note that the free and wired effective resistance $R^F$ and $R^W$ are covered by the statement of Proposition \ref{prop:LimitGraphConsistency} since
	\begin{equation}
	R^F(x,y) = \lb\inf \mengedef{\mathcal{E}(u)}{u(x) = 1, u(y) = 0, u \in \domE}\rb^{-1}
	\end{equation}
	and
	\begin{equation}
	R^W(x,y) = \lb\inf \mengedef{\mathcal{E}(u)}{u(x) = 1, u(y) = 0, u \in \Fin}\rb^{-1}
	\end{equation}
	by Theorems 2.14 and 2.20 in \cite{JorgensenPearse2009}.
\end{rem}

\section{Weak convergence of random walks}
\label{section:WeakLimitRandomWalk}

As in Section \ref{section:LimitingProperties}, we are given a resistance metric $R$ on a countably infinite set $V$ and a finite exhaustion $(V_n)_{n\in \N}$ of $V$. We denote by $G_n = (V_n, c_n)$ the unique finite graph with effective resistance $R \restr_{V_n}$ and by $G_R = (V,c)$ the associated limit graph. For $x \in V$, let $\P^n_x$ be the random walk on $G_n$ starting in $x$. We will now investigate under which conditions the sequence $(\P^n_x)$ has a weak limit point.

Prohorov's theorem states that on a separable and complete metric space $\Omega$, a sequence $(\mu_n)_{n \in \N}$ of probability measures on the Borel-$\sigma$-algebra is sequentially compact with respect to weak convergence in the space of probability measures on $\Omega$ if and only if the sequence is tight, i.e.,
\begin{equation}
\label{eq:Tightness}
\all \varepsilon > 0 \ex K_{\varepsilon} \subset \subset \Omega \all n \in \N: \mu_n(\Omega \setminus K_{\varepsilon}) < \varepsilon.
\end{equation}
Let $\Omega = V^{\N_0}$ and $d_{\text{disc}}$ be the discrete metric on $V$, i.e., $d_{\text{disc}}(x,y) = 1$ for $x\neq y$. We equip $\Omega$ with the corresponding product topology which can be achieved by choosing the metric
$$\rho(\omega, \omega') = \sum_{k=0}^{\infty} 2^{-k} \cdot d_{\text{disc}}(\omega_k, \omega'_k).$$
We leave the verification that $(\Omega, \rho)$ is complete and separable as an exercise. Although $\P^n_x$ is originally defined on $(V_n)^{\N_0}$, we can interpret it as a Borel-measure on $\Omega$ since it is uniquely defined on all cylinder sets if we think of the vertices in $V \setminus V_n$ as unreachable.

Since every subset of $V$ is open, the product topology is generated by the set of all cylinder sets
$$\mathcal{C} = \mengedef{A_1 \times \ldots \times A_n \times V^{\N}}{n \in \N, A_i \subseteq V}.$$
Note the following properties of $\mathcal{C}$.
\begin{equation}
\label{eq:CylinderSetsClosedUnderIntersection}
\all A,B \in \mathcal{C} : A \cap B  \in \mathcal{C}
\end{equation}
\begin{equation}
\label{eq:CylinderSetsComplement}
\all A \in \mathcal{C} \ex n \in \N, B_1,\ldots, B_n \in \mathcal{C}: \Omega \setminus A = B_1 \cup \ldots \cup B_n
\end{equation}
In particular, $\mathds{1}_A$ is bounded and continuous for all $A \in \mathcal{C}$.

We begin our investigation of when $(\P^n_x)_{n \in\N}$ is weakly convergent by observing a necessary condition.
\begin{lem}
	\label{lemma:InfiniteCxImliesGraveyard}
	Let $x \in V$. If $\lim_{n \to \infty} (c_n)_x = \infty$, then $\lim_{n\to \infty} \P^n_x[\omega_1 = y] = 0$ for all $y \in V$.
\end{lem}
\begin{proof}
	Let $y \in V$, $y \neq x$, and $n \in \N$. If $c_n(x,y) > 0$, we have $R(x,y) \leq c_n(x,y)^{-1}$ because $R$ is the effective resistance of $G_n$ (cf. \cite[Lemma 4.3]{JorgensenPearse2009}). Hence, we always have $c_n(x,y) \leq R(x,y)^{-1}$ and thus
	$$\P^n_x[\omega_1 = y] = \frac{c_n(x,y)}{(c_n)_x} \leq \frac{R(x,y)^{-1}}{(c_n)_x}~.$$
	It follows that $\lim_{n \to \infty} \P^n_x[\omega_1 = y] = 0$.
	
	In the case of $x = y$, we have $c_n(x,x) = 0$ and thus $\P^n_x[\omega_1 = x] = 0$ for all $n \in \N$.\qed
\end{proof}
Lemma \ref{lemma:InfiniteCxImliesGraveyard} shows that $(\P^n_x)_{n\in \N}$ can not have a weak limit if $(c_n)_x \to \infty$. Indeed, suppose that $\P_x$ is a weak limit of $(\P^n_x)_{n \in \N}$. Since $\menge{\omega_1 = y}, \menge{\omega_1\in V} \in \mathcal{C}$, it follows that $\mathds{1}_{\menge{\omega_1 = y}}$ and $\mathds{1}_{\menge{\omega_1\in V}}$ are bounded and continuous with respect to the topology of $\Omega$. Since $\mu(A) = \int \mathds{1}_A ~ \text{d}\mu$ for any measure $\mu$, we have
$$1 = \lim_{n \to \infty} \P^n_x[\omega_1 \in V] = \P_x[\omega_1 \in V] = \sum_{y \in V} \P_x[\omega_1 = y] = \sum_{y \in V}  \lim_{n \to \infty} \P^n_x[\omega_1 = y] = 0$$
which is an obvious contradiction. It follows that the following condition (\ref{eq:A}) is necessary for the tightness of $(\P_x^n)_{n \in \N}$ and we will from now on assume that it is satisfied.
\begin{equation*}
\label{eq:A}
\tag{C}
\lim_{n\to \infty} (c_n)_x < \infty \all x \in V
\end{equation*}
\begin{rem}
	Note that (\ref{eq:A}) implies that the metric space $(V,R)$ must not contain limit points. Indeed, suppose that $R(x,y_n) \to 0$ for some sequence $(y_n)_{n \in\N}$ of vertices and $x \in V$. By the probabilistic representation (\ref{eq:EffectiveResistanceProabilistic}) of $R$ in $(V_n, c_n)$, we have
	$$0 \leq \frac{1}{(c_n)_x} \leq \frac{1}{(c_n)_x} \E_x^n\left[\sum_{k=0}^{\tau_{y_n}-1} \mathds{1}_x(\omega_k)\right] = R(x,y_n) \to 0$$
	which implies $(c_n)_x \to \infty$.
\end{rem}
\begin{lem}
	\label{lemma:TightnessEquivalentT1}
	Let $x_0 \in V$. The sequence $(\P^n_{x_0})_{n \in \N}$ is tight if and only if
	\begin{equation*}
	\tag{T1}
	\label{eq:T1}
	\all \varepsilon > 0 \ex m \in \N \all n \in \N : \sum_{y \notin V_m} \frac{c_n(x,y)}{(c_n)_x} < \varepsilon
	\end{equation*}
	holds for all $x \in V$.
\end{lem}
\begin{proof}
	First, assume that (\ref{eq:T1}) holds for all $x \in V$. Then, for $\delta > 0$ and $x \in V$, there exists $m = m(x) \in \N$ such that
	$$\sup_{n \in \N} \sum_{y \notin V_m} \frac{c_n(x,y)}{(c_n)_x} < \delta ~.$$
	We define $N(x, \delta) := V_m$ where $m$ is the smallest number satisfying the above inequality. For a finite set $F \subseteq V$, we define
	$$N(F, \delta) := \bigcup_{x \in F} N(x, \nicefrac{\delta}{|F|}).$$
	Then, 
	$$\sup_{n \in \N} \sum_{x \in F} \sum_{y \notin N(F, \delta)} \frac{c_n(x,y)}{(c_n)_x} \leq  \sum_{x \in F} \left( \sup_{n \in \N} \sum_{y \notin N(x, \delta)} \frac{c_n(x,y)}{(c_n)_x} \right) < \delta.$$
	Now let $\varepsilon > 0$ and fix $x_0 \in V$. Furthermore, let $A_0 := \menge{x_0}$ and for $n \in \N_0$
	$$A_{n+1} = N(A_n, \nicefrac{\varepsilon}{2^{n+1}}).$$
	Finally, set $K_{\varepsilon} := \prod_{n\in \N_0} A_n$. Since all $A_n$ are finite, $K_{\varepsilon}$ is compact in $\Omega$. Using the Law of total probability we compute
	\begin{align*}
	\P^n_{x_0}(\Omega \setminus K_{\varepsilon}) & = \sum_{k=0}^{\infty} \P^n_{x_0}(A_0 \times \ldots \times A_{k-1} \times (V \setminus A_k) \times V^{\N})\\
	& = \sum_{k=1}^{\infty} \prod_{m=1}^{k-1} \P^n_{x_0}[\omega_m \in A_m ~ | ~ \omega_{m-1} \in A_{m-1},\ldots, \omega_0 \in A_0] \\
	& \hspace{1.5cm}\cdot \P^n_{x_0}[\omega_k \notin A_k ~ | ~ \omega_{k-1} \in A_{k-1}, \ldots, \omega_0 \in A_0]\\
	& \leq \sum_{k=1}^{\infty} \P^n_{x_0}[\omega_k \notin A_k ~ | ~ \omega_{k-1} \in A_{k-1}, \ldots, \omega_0 \in A_0]\\
	& \leq \sum_{k=1}^{\infty} \sum_{x \in A_{k-1}} \P^n_{x_0}[\omega_k \notin A_k~ | ~ \omega_{k-1} = x ]\\
	& = \sum_{k=1}^{\infty} \sum_{x \in A_{k-1}} \sum_{y \notin A_k} \frac{c_n(x,y)}{(c_n)_x}\\
	& < \sum_{k=1}^{\infty} \frac{\varepsilon}{2^k} = \varepsilon 
	\end{align*}
	It follows that $(\P^n_{x_0})_{n\in \N}$ is tight. 
	
	We use contraposition to prove that tightness implies (\ref{eq:T1}), i.e. assume that
	$$\ex \varepsilon > 0 \all m \in \N \ex n \in \N : \sum_{y \notin V_m} \frac{c_n(x,y)}{(c_n)_x} \geq \varepsilon$$
	for some $x \in V$. Furthermore let $K \subset\subset \Omega$ be any compact set. Since $\Omega = V^{\N_0}$ is equipped with the product topology of the discrete topology on $V$, the projection $\pi_1 : \Omega \to V$, $\omega\mapsto \omega_1$ is continuous. Hence, $\pi_1(K) \subseteq V$ is compact and thus finite. It follows that there exists $m \in \N$ such that $\pi_1(K) \subseteq V_m$. By our assumption above there exists $n \in \N$ such that
	$$\sum_{y \notin V_m} \frac{c_n(x,y)}{(c_n)_x} \geq \varepsilon$$. Hence,
	\begin{align*}
	\P^n_{x_0}(\Omega\setminus K) & \geq \P^n_{x_0} [\omega_1 \notin \pi_1(K)] \geq \P^n_{x_0}[\omega_1 \notin V_m]\\
	& = \sum_{y \notin V_m} \P^n_{x_0}[\omega_1 = y] = \sum_{y \notin V_m} \frac{c_n(x,y)}{(c_n)_x} \geq \varepsilon.
	\end{align*}
	This implies that the sequence $(\P^n_{x_0})_{n \in \N}$ is not tight.\qed
\end{proof}
\begin{lem}
	\label{lemma:T1equivalentT2}
	Let $x \in V$. Under the assumption (\ref{eq:A}), the condition (\ref{eq:T1}) is equivalent to 
	\begin{equation*}
	\tag{T2}
	\label{eq:T2}
	\lim_{n \to \infty} (c_n)_x = c_x.
	\end{equation*}
\end{lem}
\begin{proof}
	Let $\varepsilon > 0$ and assume that (\ref{eq:T2}) holds. Then, 
	$$\lim_{n \to \infty} \frac{c_n(x,y)}{(c_n)_x} = \frac{c(x,y)}{c_x}$$
	and there exists $m \in \N$ such that 
	$$1 - \sum_{y \in V_m} \frac{c(x,y)}{c_x} < \varepsilon.$$
	Let $n \in \N$. For $n \leq m$ we have $V_n \subseteq V_m$ and thus
	$$\sum_{y \notin V_m} \frac{c_n(x,y)}{(c_n)_x} = 1 - \sum_{y \in V_m} \frac{c_n(x,y)}{(c_n)_x} = 0.$$
	For $n > m$ we use the monotony of the quotients (see Corollary \ref{cor:MonotonyOfEdgeWeights}) and get
	$$\sum_{y \notin V_m} \frac{c_n(x,y)}{(c_n)_x} = 1 - \sum_{y \in V_m} \frac{c_n(x,y)}{(c_n)_x} \leq 1 - \sum_{y \in V_m} \frac{c(x,y)}{c_x} < \varepsilon.$$
	Hence, we have proven that (\ref{eq:T1}) holds.
	
	Now assume that (\ref{eq:T1}) holds. By Corollary \ref{cor:MonotonyOfEdgeWeights} and (\ref{eq:A}), we have
	$$c_x = \lim_{n \to \infty} \sum_{y \in V_n} c(x,y) \leq \lim_{n\to \infty} \sum_{y \in V_n} c_n(x,y) = \lim_{n \to \infty} (c_n)_x <\infty.$$
	Let $\varepsilon > 0$. Then the convergence of the series $\sum_{y \in V} c(x,y)$ implies that there exists $m_1 \in \N$, such that
	$$\sum_{y \notin V_{m_1}} c(x,y)  < \frac{\varepsilon}{3}~.$$
	Furthermore, let $D_x := \lim_{n \to \infty} (c_n)_x$ and choose $m_2 \in \N$ such that 
	$$\sum_{y \notin V_{m_2}} \frac{c_n(x,y)}{(c_n)_x} < \frac{\varepsilon}{3D_x}$$
	for all $n \in \N$. Now, let $m = \max(m_1,m_2)$ and let $n \in \N$ such that 
	$$\left\vert \sum_{y \in V_m} \lb c_n(x,y) - c(x,y) \rb\right\vert < \frac{\varepsilon}{3}~.$$
	This is possible due to the fact that $c_n(x,y) \to c(x,y)$ and that $V_m$ is finite. It follows that
	\begin{align*}
	\left\vert \sum_{y \in V_n} \lb c_n(x,y) - c(x,y) \rb \right\vert & \leq \left\vert \sum_{y \in V_m} \lb c_n(x,y) - c(x,y) \rb \right\vert + \left\vert \sum_{y \notin V_m} \lb c_n(x,y) - c(x,y) \rb \right\vert\\
	& < \frac{\varepsilon}{3} + \left\vert \sum_{y \notin V_{m_2}} c_n(x,y) \right\vert + \left\vert \sum_{y \notin V_{m_1}} c(x,y)\right\vert\\
	& < \frac{\varepsilon}{3} +(c_n)_x \cdot \frac{\varepsilon}{3D_x} + \frac{\varepsilon}{3} \leq \varepsilon
	\end{align*}
	The last inequality is due to Corollary \ref{cor:MonotonyOfEdgeWeights} which implies that 
	$$D_x = \lim_{m \to \infty} (c_m)_x \geq (c_n)_x$$ 
	for all $n \in \N$. Hence, we have shown that
	\begin{equation*}
	\lim_{n\to \infty} (c_n)_x = \lim_{n \to \infty} \sum_{y \in V_n} c_n(x,y) = \lim_{n\to \infty} \sum_{y \in V_n} c(x,y) = \sum_{y \in V} c(x,y). \tag*{\qed}
	\end{equation*}
\end{proof}

\begin{theorem}
\label{theorem:WeakConvergence}
	Fix $x_0 \in V$. The sequence of random walks $(\P^n_{x_0})_{n \in \N}$ on $(V_n, c_n)$ starting in $x_0$ is sequentially compact in the set of probability measures on $V^{\N_0}$ if and only if
	\begin{equation}
	\label{eq:TightnessCondition}
	\sum_{y \in V} \lim_{n \to \infty} c_n(x,y) = \lim_{n \to \infty} \sum_{y \in V_n} c_n(x,y) < \infty
	\end{equation}
	for all $x \in V$.
\end{theorem}
\begin{proof}
	The claim follows by Prokhorov's theorem, Lemma \ref{lemma:TightnessEquivalentT1} and Lemma \ref{lemma:T1equivalentT2}. \qed
\end{proof}

The following example shows that the condition (\ref{eq:TightnessCondition}) for tightness is not automatically satisfied and thus can not be omitted.

\begin{exa}
	\label{ex:NoWeakConvergence}
	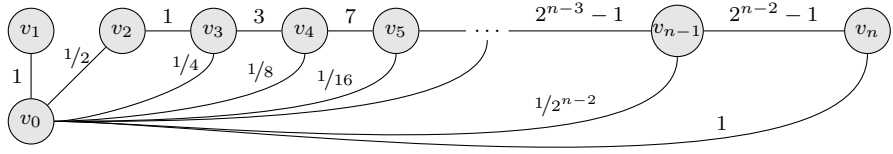
\begin{figure}
		\centering
		\tikzstyle{perp} = [out=0, in=270, looseness = 0.5]
		\tikzstyle{place} = [above, pos = 0.7]
		\begin{tikzpicture}[-,auto, node distance = 1.2cm, every loop/.style={}, vertex/.style={draw, circle, fill=black!10, inner sep=0mm, minimum size=6mm}]
		\node[vertex] (A) {$v_0$};
		\node[vertex, above of= A] (B) {$v_1$};
		\node[vertex,right of =B] (C) {$v_2$};
		\node[vertex,right of=C] (D) {$v_3$};
		\node[vertex,right of=D] (E) {$v_4$};
		\node[vertex,right of=E] (F) {$v_5$};
		\node[right of=F] (G) {$\ldots$};
		\node[vertex, right of=G, node distance = 2.5cm] (H) {$v_{n-1}$};
		\node[vertex, right of=H, node distance = 2.5cm] (I) {$v_n$};
		
		\path 	(A) edge node {$1$} (B)
		(A) edge node[above] {$\nicefrac{1}{2}$} (C)
		(A) edge [perp] node [place] {$\nicefrac{1}{4}$} (D)
		(A) edge [perp] node [place] {$\nicefrac{1}{8}$} (E)
		(A) edge [perp] node [place] {$\nicefrac{1}{16}$} (F)
		(A) edge [perp] (G)
		(A) edge [perp] node [place] {$\nicefrac{1}{2^{n-2}}$} (H)
		(A) edge [perp] node [place] {$1$} (I)
		(C) edge node {$1$} (D)
		(D) edge node {$3$} (E)
		(E) edge node {$7$} (F)
		(F) edge (G)
		(G) edge node {$2^{n-3} - 1$} (H)
		(H) edge node {$2^{n-2} - 1$} (I);
		\end{tikzpicture}
		\caption{\label{fig:ExampleTightness}A sequence of graphs not satisfying condition (\ref{eq:TightnessCondition}).}
	\end{figure}
	Consider the following sequence of graphs, denoted by $G_n = (V_n, c_n)$ where $c_n$ is given by
	\begin{align*}
	c_n(v_0,v_k) & = c_n(v_k, v_0) =  \begin{cases} \nicefrac{1}{2^{k-1}} & , ~ 1 \leq k < n\\
	1 & , ~ k = n
	\end{cases},\\
	c_n(v_k,v_{k+1}) & = c(v_{k+1}, v_k) = 2^{k-1} - 1 \text{ for } 1 \leq k \leq n-1
	\end{align*}
	and $c_n(x,y) = 0$ otherwise, see Figure \ref{fig:ExampleTightness}. Let $R_n$ be the effective resistance of $G_n$. We verify that $R_n$ agrees with $R_{n-1}$ on $V_{n-1}$ by applying the star-mesh transform (Lemma \ref{lemma:StarMesh}) to $v_n$ in $G_n$. This yields a graph $(V_{n-1}, c')$ with effective resistance $R_n \restr_{V_{n-1}}$ and
	$$c'(x,y) = c_n(x,y) + \frac{c_n(x,v_n)c_n(v_n,y)}{c_{v_n}}.$$
	Hence, for all $x,y \in V_{n-1}$ such that $\menge{x,y} \neq \menge{v_0,v_{n-1}}$, we have $c'(x,y) = c_n(x,y) = c_{n-1}(x,y)$. Furthermore,
	$$c'(v_0, v_{n-1}) = \frac{1}{2^{n-2}} + \frac{1 \cdot (2^{n-2} - 1)}{1 + 2^{n-2} - 1} = \frac{2^{n-2}}{2^{n-2}} = 1 = c_{n-1}(v_0, v_{n-1}).$$
	It follows that $c' = c_n$ and thus $R_{n-1} = R_n \restr_{V_{n-1}}$.
	
	Now we will see that $(G_n)_{n \in \N}$ does not satisfy (\ref{eq:TightnessCondition}). We have
	$$c(v_0, v_k) = \lim_{n \to\infty} c_n(v_0, v_k) = \frac{1}{2^{k-1}}$$
	for all $k \in \N$ and thus
	$$c_{v_0} = \sum_{y \in V} c(v_0,y) = \sum_{k = 1}^{\infty} \frac{1}{2^{k-1}} = 2.$$
	However, 
	$$\lim_{n \to \infty} (c_n)_{v_0} =\lim_{n \to \infty}  \sum_{y \in V_n} c_n(v_0,y) = \lim_{n \to \infty}  \left( 1 + \sum_{k=1}^{n-1} \frac{1}{2^{k-1}}\right) = 3.$$ 
\end{exa}
\begin{rem}
	Note that Example \ref{ex:NoWeakConvergence} can be modified such that $c_n(v_0, v_n) = n$ and 
	$$c_n(v_k, v_{k+1}) = \frac{2^{k-1}(k+1)^2}{2^{k-1}+1} - (k+1).$$ 
	This will still yield a sequence of graphs with invariant effective resistances and $c_{v_0} = 2$. However, in this case $\lim_{n \to\infty} (c_n)_{v_0} \geq n \to \infty$.
\end{rem}

In the case of tightness, Prohorov's theorem only implies the existence of a weak limit point of $(\P_x^n)_{n \in \N}$. In our context, we can obtain the weak convergence of the whole sequence and identify the weak limit as random walk of the limit graph $G_R$.
\begin{theorem}
	\label{theorem:WeakLimitIsRandomWalkOnLimitGraph}
	If $c_x = \lim_{n \to \infty} (c_n)_x < \infty$ for all $x \in V$, then $(\P_x^n)_{n \in\infty}$ converges weakly to the random walk $\P_x$ of the limit graph $G_R$ for all $x \in V$.
\end{theorem}
\begin{proof}
	Let $x \in V$. By Theorem \ref{theorem:WeakConvergence}, $(\P_x^n)_{n \in \N}$ is tight. Hence, there exists a subsequence $(\P_x^{n_k})_{k \in \N}$ which weakly converges to a measure $\mu$. For $x_0, \ldots, x_n \in V$, let $A := \menge{x_0} \times \ldots\times \menge{x_n} \times V^{\N} \in \mathcal{C}$. Then, $\mathds{1}_A$ is bounded and continuous and it follows that
	\begin{align*}
	\mu[\omega_0 = x_0, \ldots, \omega_n = x_n] & = \int \mathds{1}_A ~ \text{d}\mu = \lim_{k\to\infty} \int \mathds{1}_A ~ \text{d}\P_x^{n_k}\\
	& = \lim_{n \to \infty} \P_x^{n_k}[\omega_0 = x_0, \ldots,\omega_n = x_n] \\
	& = \lim_{n \to \infty} \mathds{1}_{x}(x_0) \cdot \prod_{l=0}^{n-1} \frac{c_{n_k}(x_l, x_{l+1})}{(c_{n_k})_{x_l}} \\
	& = \mathds{1}_{x}(x_0) \cdot \prod_{l=0}^{n-1} \frac{c(x_l, x_{l+1})}{c_{x_l}} = \P_x[\omega_0 = x_0, \ldots,\omega_n = x_n].
	\end{align*}
	For $C \in \mathcal{C}$, there exist $A_0,\ldots, A_n \subseteq V$ such that $C = A_0\times \ldots\times A_n \times V^{\N}$. Note that $A_0 \times \ldots \times A_n$ is at most countably infinite. Since
	$$C = \bigcup_{(x_0,\ldots,x_n) \in A_0 \times \ldots \times A_n} \menge{x_0} \times \ldots \times \menge{x_n} \times V^{\N}$$
	and this is a disjoint union, the $\sigma$-additivity of $\mu$ and $\P_x$ implies that $\mu(C) = \P_x(C)$. Hence, $\mu$ and $\P_x$ agree on $\mathcal{C}$.
	Since $(\Omega, \rho)$ is a countable product of separable metric spaces, its Borel-$\sigma$-algbra $\mathcal{B}$ satisfies
	$$\mathcal{B} = \mathcal{B}(V) \otimes \mathcal{B}(V) \otimes \ldots = \sigma(\mathcal{C}).$$
	By (\ref{eq:CylinderSetsClosedUnderIntersection}), $\mathcal{C}$ is closed under finite intersections. It is a standard result that $\mu \restr_{\mathcal{C}} = \P_x\restr_{\mathcal{C}}$ implies $\mu = \P_x$. 
	
	Hence, we have shown that every weakly convergent subsequence of $(\P_x^n)_{n \in \N}$ converges to $\P_x$. Since every subsequence of a tight sequence is again tight, we obtain the following statement. Every subsequence of $(\P_x^n)_{n \in \N}$ contains a subsequence converging to $\P_x$. This is equivalent to $(\P_x^n)_{n \in \N}$ converging to $\P_x$.\qed
\end{proof}
By (\ref{eq:EffectiveResistanceProabilistic}), we have for $x,y \in V_n$,
\begin{equation}
\label{eq:EffectiveResistanceProbabilisticSequence}
R(x,y) = \frac{1}{(c_n)_x} \int \sum_{k=0}^{\tau_y-1} \mathds{1}_x(\omega_k) ~ \text{d} \P_x^n
\end{equation}
for all $n \in \N$. Since $(c_n)_x \to c_x$ and $\P^n_x$ converges weakly to $\P_x$, one may hope to get an analogous equation in terms of $c_x$ and $\P_x$ which would then hold for all $x,y \in V$. The following example will show that this is in general false.

\begin{exa}
	We consider two sequences of graphs $G_n = (V_n, c_n)$ and $H_n = (V_n, \hat{c}_n)$ with $V_n = \mengedef{z \in \Z}{|z| \leq n}$, 
	$$c_n(x,y) = \begin{cases}
	0 & , |x-y| \neq 1\\
	2^{\min(|x|,|y|)} & , |x-y| = 1
	\end{cases},$$
	see Figure \ref{fig:ExampleWeakConvergenceFail1}, and $\hat{c}_n(x,y) = c_n(x,y) + \mathds{1}_{\menge{-n,n}}(\menge{x,y}) \cdot 2^{n-2}$, see Figure \ref{fig:ExampleWeakConvergenceFail2}. 
	\begin{figure}
	\centering
	\begin{tikzpicture}[-,auto, node distance = 1.4cm, every loop/.style={}, vertex/.style={draw, circle, fill=black!10, inner sep=0mm, minimum size=7mm}]
		\node[vertex] (A) {$-n$};
		\node[right of=A] (B) {$\ldots$};
		\node[vertex, right of=B] (C) {$-2$};
		\node[vertex, right of=C] (D) {$-1$};
		\node[vertex, right of=D] (E) {$0$};
		\node[vertex, right of=E] (F) {$1$};
		\node[vertex, right of=F] (G) {$2$};
		\node[right of=G] (H) {$\ldots$};
		\node[vertex, right of=H] (I) {$n$};
		
		\path (A) edge node {$2^{n-1}$} (B)
			(B) edge node {$4$} (C)
			(C) edge node {$2$} (D)
			(D) edge node {$1$} (E)			
			(E) edge node {$1$} (F)
			(F) edge node {$2$} (G)
			(G) edge node {$4$} (H)			
			(H) edge node {$2^{n-1}$} (I);		
	\end{tikzpicture}
	\caption{\label{fig:ExampleWeakConvergenceFail1}The graph $G_n$.}
\end{figure}
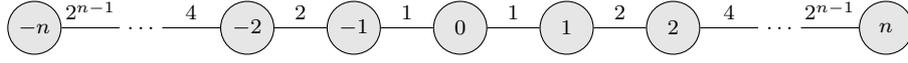
\begin{figure}
	\centering
	\begin{tikzpicture}[-,auto, node distance = 1.4cm, every loop/.style={}, vertex/.style={draw, circle, fill=black!10, inner sep=0mm, minimum size=7mm}]
	\node[vertex] (A) {$-n$};
	\node[right of=A] (B) {$\ldots$};
	\node[vertex, right of=B] (C) {$-2$};
	\node[vertex, right of=C] (D) {$-1$};
	\node[vertex, right of=D] (E) {$0$};
	\node[vertex, right of=E] (F) {$1$};
	\node[vertex, right of=F] (G) {$2$};
	\node[right of=G] (H) {$\ldots$};
	\node[vertex, right of=H] (I) {$n$};
	
	\path (A) edge node {$2^{n-1}$} (B)
	(B) edge node {$4$} (C)
	(C) edge node {$2$} (D)
	(D) edge node {$1$} (E)			
	(E) edge node {$1$} (F)
	(F) edge node {$2$} (G)
	(G) edge node {$4$} (H)			
	(H) edge node {$2^{n-1}$} (I)
	(A) edge[out=65, in=115, looseness = 0.2] node {$2^{n-2}$} (I);		
	\end{tikzpicture}
	\caption{\label{fig:ExampleWeakConvergenceFail2}The graph $H_n$.}
\end{figure}
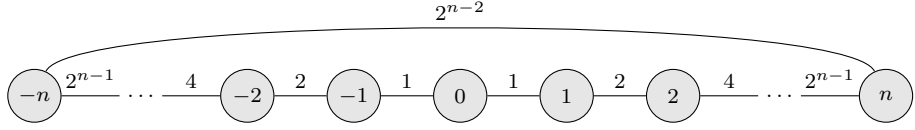
Let $n \in \N$. Note that $G_n$ and $H_n$ are respectively obtained from $G_{n+1}$ and $H_{n+1}$ by applying the star-mesh transform (see Lemma \ref{lemma:StarMesh}) to the vertices $n+1$ and $-(n+1)$ successively. Hence, $R_{G_{n+1}}(x,y) = R_{G_n}(x,y)$ and $R_{H_{n+1}}(x,y) = R_{H_n}(x,y)$ for all $x,y \in V_n$. It follows that $R_{G_n}$ and $R_{H_n}$ are finite restrictions of resistance metrics $R_1$ and $R_2$ on $V$. 

Since $G_n$ is a tree, its effective resistance equals its geodesic metric. Let $f(x) := \sum_{k=0}^{|x|-1} 2^{-k}$. Then, for $x,y \in V_n$, we have
$$R_1(x,y) = R_{G_n}(x,y) = \begin{cases}
f(x) + f(y) & , ~ x \cdot y \leq 0\\
f(\max(|x|, |y|)) - f(\min(|x|,|y|)) & , ~ x \cdot y  > 0
\end{cases}.$$
Note that for $x > 0$, $R_1(x,n) = R_1(0,n) - R_1(0,x) = 2 - 2^{-(n-1)} - R_1(0,x)$. 

Let $x,y \in V_n$ and w.l.o.g. assume that $x < y$. Using basic network reduction rules for parallel and series circuits, one obtains
\begin{align*}
R_2(x,y)  & = \lb \frac{1}{R_1(x,y)} + \frac{1}{R_1(x,-n) + 2^{-(n-2)} + R_1(y,n)} \rb^{-1} \\
& = \lb \frac{1}{R_1(x,y)} + \frac{1}{4 - R_1(x,y)} \rb^{-1} = \frac{1}{4}R_1(x,y)(4-R_1(x,y)).\\
\end{align*}
Hence, $R_1$ and $R_2$ are different resistance metrics. However, both limit graphs are the same, namely $G = (\Z, c)$ where $c(x,y) = c_n(x,y)$ if $x,y \in V_n$. It follows that both sequences of random walks on $G_n$ and on $H_n$ have the same weak limit with respect to our chosen topology. Hence, it is impossible to write both $R_1$ and $R_2$ in terms of the random walk on $G$ as in (\ref{eq:EffectiveResistanceProbabilisticSequence}).
\end{exa}
\begin{rem}
	$R_1$ and $R_2$ are the free and wired effective resistance of $G$, respectively. In view of Proposition \ref{prop:LimitGraphConsistency}, the above example is not surprising since $R_1 \neq R_2$ but both are resistance metrics induced by the same graph. Hence, it is impossible to find a probabilistic representation which fits $R_1$ and $R_2$ simultaneously but only uses the random walk on $G$. 
	
	Note that for the free and wired effective resistance not to be equal, it is necessary that the graph is transient. As we will see below, one can produce a probabilistic representation for $R$ if its limit graph is recurrent.
\end{rem}

From now on assume that $\P^n_x$ converges weakly to $\P_x$. We use the notation $\P^n_x \Rightarrow \P_x$. By definition, this gives us
$$\lim_{n \to \infty} \E^n_x[ f  ]  = \E_x[ f ]$$
for all continuous and bounded $f: \Omega \to \R$. Note that, by the Portmanteau theorem, this can be extended to measurable and bounded $f: \Omega \to \R$ whose set of points of discontinuity is a null set of $\P_x$.
\begin{dfn}
	For $x,y \in V$, $x\neq y$, let $\Phi^{xy}: V^{\N_0} \to \R$ be defined by
	$$\Phi^{xy}(\omega) = \sum_{k=0}^{\tau_y-1} \mathds{1}_x(\omega_k).$$
	For $x = y$, let $\Phi^{xy} \equiv 0$.
\end{dfn}
$\Phi^{xy}$ is continuous only on $\menge{\tau_y < \infty}$ and not bounded. Hence, we can not directly apply weak convergence to obtain $\E^n_x[\Phi^{xy}] \to \E_x[\Phi^{xy}]$. To deal with the unboundedness, we first compute the distribution of $\Phi^{xy}$ with respect to $\P^n_x$.
\begin{lem}
	\label{lemma:PhiGeometricDistribution}
	For $x,y \in V$, $x \neq y$, and $n$ large enough such that $x,y \in V_n$, there exists $p_n \in [0,1]$ such that 
	$$\P_x^n[\Phi^{xy} = 0] = 0$$
	and
	$$\P^n_x[\Phi^{xy} = k] = (1-p_n)p_n^{k-1}$$
	for all $k \in \N^+$. More precisely, we have $p_n = \P^n_x[\tau_x^+ < \tau_y]$.
\end{lem}
\begin{proof}
	The proof is a straightforward computation using induction over $k$ and the Markov property of $\P^n_x$. \qed
\end{proof}
Knowing the exact distribution of $\Phi^{xy}$, we can compute its first and second moment with respect to $\P^n_x$. Because all graphs $(V,c_n)$ are connected and finite, we have $p_n < 1$ and $\P_x^n[\tau_y < \infty] = 1$ for all $n \in \N$. Hence, we see that
\begin{equation}
\label{eq:FirstMomentPhi}
\E^n_x[ \Phi^{xy} ] = \sum_{k=1}^{\infty} k (1-p_n)p_n^{k-1} = \frac{1}{1-p_n}~.
\end{equation}
For the second moment, we have
\begin{equation}
\label{eq:SecondMomentPhi}
\E_x^n\left[ (\Phi^{xy})^2 \right] = \sum_{k=1}^{\infty} k^2 (1-p_n)p_n^{k-1} = \frac{1+p_n}{(1-p_n)^2}~.
\end{equation}
\begin{rem}
If we inserts (\ref{eq:FirstMomentPhi}) into (\ref{eq:EffectiveResistanceProbabilisticSequence}), we get 
$$R_n(x,y) = \frac{1}{(c_n)_x \P^n_x[\tau_y \leq \tau_x^+]}~.$$
This is a well-known formula for the effective resistance of finite graphs (see \cite[Section 2.3]{LyonsPeres2015}) and is often written in the form
$$R_n(x,y) = \frac{1}{c_x \P^n_x[x \to y]}~.$$
\end{rem}
\begin{theorem}	
	\label{theorem:ProbabilisticRepresentationMetricInfinite}
	Let $(V,R)$ be a countably infinite resistance metric space, $(V_n)_{n \in \N}$ a finite exhaustion of $V$, $G_n = (V_n, c_n)$ finite graphs with effective resistance $R \restr_{V_n}$ and $G_R = (V, c)$ be the associated limit graph. If
	$$c_v = \lim_{n \to \infty} (c_n)_v < \infty$$ 
	for all $v \in V$, and the random walk of $G_R$ is recurrent, then, for $x,y \in V$, we have
	$$R(x,y) = \frac{1}{c_x} \E_x\left[\sum_{k=0}^{\tau_y - 1} \mathds{1}_x(\omega_k) \right].$$
\end{theorem}
\begin{proof}
	For $x = y$, both $R(x,y)$ and the equation's right-hand side vanish. 
	
	Now assume $x \neq y$. By Theorem \ref{theorem:WeakLimitIsRandomWalkOnLimitGraph}, we know that $\P_x^n \Rightarrow \P_x$.
	For simplicity, let $\Phi:= \Phi^{xy}$. Throughout this proof assume that all $n$ are large enough such that $x,y \in V_n$. By (\ref{eq:EffectiveResistanceProbabilisticSequence}) and (\ref{eq:FirstMomentPhi}), we have
	$$(c_n)_x \cdot R(x,y) = \E_x^n [ \Phi ] = \frac{1}{1-p_n}~.$$
	Since $\lim_{n \to \infty} (c_n)_x < \infty$, it follows that $p_n$ is uniformly bound away from $1$ and thus
	$$\E^n_x[ \Phi^2 ] = \frac{1+p_n}{(1-p_n)^2} \leq K$$
	for some $K > 0$ and all sufficiently large $n$. Hence, for $T > 0$, we have
	$$\int_{\menge{\Phi > T}} \Phi ~ \text{d}\P^n_x \leq \frac{1}{T} \E_x^n[ \Phi^2 ] \leq \frac{K}{T}~.$$
	
	Let $\varepsilon > 0$. We define $\Phi^T(\omega) = \min(\Phi(\omega), T))$. Then, $\Phi^T \nearrow \Phi$ point-wise as $T \to \infty$. By the Theorem of monotone convergence, there exists $T_1$ such that 
	$$\E_x\left[\Phi - \Phi^T\right] \leq \frac{\varepsilon}{3}$$
	for all $T \geq T_1$. For $T \geq \max(T_1, \nicefrac{3K}{\varepsilon})$, we then also have
	$$\E_x^n\left[ \Phi - \Phi^T \right] = \int_{\menge{\Phi > T}} \Phi - T ~ \text{d}\P^n_x \leq \int_{\menge{\Phi > T}} \Phi ~ \text{d}\P^n_x \leq \frac{K}{T} \leq \frac{\varepsilon}{3}$$
	for all $n$. Since $\P_x$ is recurrent, we have $\P_x[\tau_y < \infty] = 1$. Hence, $\Phi^T$ is continuous $\P_x$-almost everywhere and bounded. It follows that $\E_x^n[\Phi^T] \rightarrow \E_x[\Phi^T]$ by $\P_x^n \Rightarrow \P_x$ and thus there exists $n_0 \in \N$ (dependent on $T$) such that
	$$\left\lvert \E_x\left[ \Phi^T \right]- \E^n_x\left[ \Phi^T \right] \right\rvert \leq \frac{\varepsilon}{3}$$
	for all $n \geq n_0$. Hence, 
	$$\left\lvert \E_x [\Phi] - \E^n_x[\Phi] \right\rvert \leq \left\lvert \E_x[\Phi - \Phi^T] \right\rvert + \left\lvert \E_x[\Phi^T] - \E^n_x[\Phi^T]\right\rvert  + \left\lvert \E^n_x[\Phi - \Phi^T ]\right\rvert \leq \varepsilon.$$
	It follows that $\E^n_x[\Phi] \to \E_x[\Phi]$ as $n \to \infty$. The claim follows from $\lim_{n\to\infty} (c_n)_x = c_x$ and
	$$R(x,y) = \frac{1}{(c_n)_x} \E^n_x\left[ \Phi^{xy} \right]$$
	by (\ref{eq:PotentialTetali}).\qed
\end{proof}
\begin{rem}
Combining Theorem 1.31 and Lemma 2.61 in \cite{Barlow2017}, it follows that the effective resistance of $G_R$ equals
\begin{equation}
\label{eq:ProbalisticRepresentationBarlow}
R_{G_R}(x,y) =  \frac{1}{c_x \P_x[\tau_y \leq \tau_x^+]} = \E_x\left[\sum_{k=0}^{\tau_y - 1} \mathds{1}_x(\omega_k) \right]
\end{equation}
if $G_R$ is recurrent. This differs from the above statement because it is not known that $R$ is the effective resistance of $G_R$. However, it can be used to identify $R$ as the effective resistance of $G_R$.
\end{rem}
\begin{cor}
\label{cor:EffResOfRecurrentLimitGraph}
Assume the same situation as in Theorem \ref{theorem:ProbabilisticRepresentationMetricInfinite}. If
$$c_v = \lim_{n \to \infty} (c_n)_v < \infty$$
holds for all $v \in V$, and $G_R$ is recurrent, then $R$ is the effective resistance of $G_R$.
\end{cor}
\begin{exa}
Our last example will show that even if $R^F = R^W$ on a transient graph, one can not expect to find a probabilistic representation of the effective resistance as in (\ref{eq:ProbalisticRepresentationBarlow}). This seems to contradict the statement of Corollary 3.13 combined with Corollary 3.15 in \cite{JorgensenPearse2009} which claims that
$$R^F(x,y) = \frac{1}{c_x \P_x[\tau_y < \tau_x^+]}~$$
on any transient graph.

Consider the graph $\mathcal{T}$ shown in Figure \ref{fig:TGraph}. It is transient and has only constant harmonic functions which implies $R^F = R^W$. Furthermore, we have $R^F(B,T) = 2$. However, 
\begin{align*}
\P_B[\tau_T < \tau_B^+] & = \P_0[\tau_T < \tau_B]\\
& = 1 - \P_0[\tau_B \leq \tau_T]\\
& = 1 - \P_0[\tau_B < \tau_T] - \P_0[\tau_B = \tau_T = \infty].
\end{align*}
Due to the symmetry of $\mathcal{T}$ we have $ \P_0[\tau_B < \tau_T] =  \P_0[\tau_T < \tau_B]$. Together with the transience of $\mathcal{T}$, this implies
$$\P_B[\tau_T < \tau_B^+] = \P_0[\tau_T < \tau_B] = \frac{1 -  \P_0[\tau_B = \tau_T= \infty]}{2} < \frac{1}{2}$$
and
$$\P_B[\tau_T \leq \tau_B^+] = \P_0[\tau_T \leq \tau_B] = \frac{1 +  \P_0[\tau_B = \tau_T= \infty]}{2} > \frac{1}{2}.$$
More precisely, one can compute 
$$\P_B[\tau_T < \tau_B^+] = \frac{2}{5} \text{ and } \P_B[\tau_T \leq \tau_B^+] = \frac{3}{5}~.$$
Hence,
$$\frac{1}{c_B \P_B[\tau_T < \tau_B^+]} \neq R^F(B,T)$$
and
$$\frac{1}{c_B}\E_B\left[\sum_{k=0}^{\tau_T - 1} \mathds{1}_B(\omega_k) \right] = \frac{1}{c_B\P_B[\tau_T \leq \tau_B^+]} \neq R^F(B,T).$$
\begin{figure}
	\centering
	\begin{tikzpicture}[-,auto, node distance = 1.2cm, every loop/.style={}, vertex/.style={draw, circle, fill=black!10, inner sep=0mm, minimum size=6mm}]
	\node[vertex] (B1) {$B$};
	\node[vertex, above of= B1] (A) {$0$};
	\node[vertex, above of= A] (T) {$T$};
	\node[vertex, right of =A] (B) {$1$};
	\node[vertex, right of=B] (C) {$2$};
	\node[vertex, right of=C] (D) {$3$};
	\node[vertex, right of=D] (E) {$4$};
	\node[right of=E] (F) {$\ldots$};
	
	\path 	(B1) edge node {$1$} (A)
	(T) edge node[left] {$1$} (A)
	(A) edge node {$1$} (B)
	(B) edge node {$2$} (C)
	(C) edge node {$4$} (D)
	(D) edge node {$8$} (E)
	(E) edge node {$16$} (F);
	\end{tikzpicture}
	\caption{\label{fig:TGraph}The transient graph $\mathcal{T}$}
\end{figure}
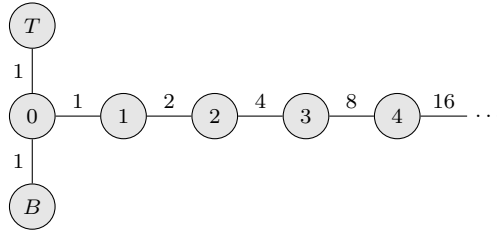	
\end{exa}

\section{Acknowledgements}
The author would like to thank Max von Renesse, Stefan Bachmann, Stefan Dück and Erik Paul for many helpful discussions.

\bibliographystyle{abbrv}
\bibliography{refs_graphs1}

\begin{thebibliography}{10}

\bibitem{AldousFill2002}
D.~Aldous and J.~Fill.
\newblock {\em Reversible Markov chains and random walks on graphs}.
\newblock Berkeley, 2002.

\bibitem{Barlow2017}
M.~T. Barlow.
\newblock {\em Random walks and heat kernels on graphs}, volume 438 of {\em
  London Mathematical Society Lecture Note Series}.
\newblock Cambridge University Press, Cambridge, 2017.

\bibitem{Chandra1997}
A.~K. Chandra, P.~Raghavan, W.~L. Ruzzo, R.~Smolensky, and P.~Tiwari.
\newblock The electrical resistance of a graph captures its commute and cover
  times.
\newblock {\em Computational Complexity}, 6(4):312--340, 1997.

\bibitem{Chung1997}
F.~R.~K. Chung.
\newblock {\em {Spectral graph theory}}, volume~92 of {\em {CBMS Regional
  Conference Series in Mathematics}}.
\newblock American Mathematical Society, Providence, RI, 1997.

\bibitem{DorflerBullo2013}
F.~Dorfler and F.~Bullo.
\newblock Kron reduction of graphs with applications to electrical networks.
\newblock {\em IEEE Transactions on Circuits and Systems I: Regular Papers},
  60(1):150--163, Jan 2013.

\bibitem{DoyleSnell1984}
P.~G. Doyle and J.~L. Snell.
\newblock {\em Random walks and electric networks}, volume~22 of {\em Carus
  Mathematical Monographs}.
\newblock Mathematical Association of America, Washington, DC, 1984.

\bibitem{JorgensenPearse2009}
P.~E.~T. Jorgensen and E.~P.~J. Pearse.
\newblock A {H}ilbert space approach to effective resistance metric.
\newblock {\em Complex Analysis and Operator Theory}, 4(4):975--1013, nov 2009.

\bibitem{Kasue2010}
A.~Kasue.
\newblock Convergence of metric graphs and energy forms.
\newblock {\em Rev. Mat. Iberoam.}, 26(2):367--448, 2010.

\bibitem{Kigami2001}
J.~Kigami.
\newblock {\em Analysis on fractals}, volume 143 of {\em Cambridge Tracts in
  Mathematics}.
\newblock Cambridge University Press, Cambridge, 2001.

\bibitem{Klein1993}
D.~J. Klein and M.~Randi{\'{c}}.
\newblock Resistance distance.
\newblock {\em Journal of Mathematical Chemistry}, 12(1):81--95, 1993.

\bibitem{Kumagai2014}
T.~Kumagai.
\newblock {\em Random walks on disordered media and their scaling limits},
  volume 2101 of {\em Lecture Notes in Mathematics}.
\newblock Springer, Cham, 2014.
\newblock Lecture notes from the 40th Probability Summer School held in
  Saint-Flour, 2010, \'Ecole d'\'Et\'e de Probabilit\'es de Saint-Flour.

\bibitem{LyonsPeres2015}
R.~Lyons and Y.~Peres.
\newblock {\em Probability on Trees and Networks}.
\newblock Cambridge University Press, New York, 2016.
\newblock Available at \url{http://pages.iu.edu/~rdlyons/}.

\bibitem{Soardi1994}
P.~M. Soardi.
\newblock {\em Potential Theory on Infinite Networks}.
\newblock Springer, 1994.

\bibitem{Spielman2007}
D.~A. Spielman.
\newblock Spectral graph theory and its applications.
\newblock In {\em Foundations of Computer Science, 2007. FOCS'07. 48th Annual
  IEEE Symposium on}, pages 29--38. IEEE, 2007.

\bibitem{Tetali1991}
P.~{Tetali}.
\newblock {Random walks and the effective resistance of networks.}
\newblock {\em {J. Theor. Probab.}}, 4(1):101--109, 1991.

\bibitem{Tutte2001}
W.~T. Tutte.
\newblock {\em Graph Theory}.
\newblock Cambridge University Press, 2001.

\end{thebibliography}

\end{document}